\documentclass[11pt,twoside, leqno]{article}

\usepackage{amsthm}
\usepackage{amssymb}
\usepackage{amsmath}
\usepackage{mathrsfs}
\usepackage{txfonts}
\usepackage{graphics}

\usepackage{hyperref}
\usepackage[Symbol]{upgreek}

\usepackage{extarrows}
\usepackage{graphics}
\usepackage{epsfig}
\usepackage{color}
\usepackage{verbatim}
\usepackage[mathlines,pagewise]{lineno}
\usepackage{amssymb,amsmath,amsfonts,amsthm,color,mathrsfs}
\usepackage[Symbol]{upgreek}
\usepackage{txfonts}
\usepackage[nottoc,notlot,notlof]{tocbibind}
\usepackage[active]{srcltx}
\usepackage{picinpar,graphicx} 
\usepackage{bbm}

\allowdisplaybreaks 
\allowdisplaybreaks
\def\rr{{\mathbb R}}
\def\rn{{{\rr}^n}}

\def\zz{{\mathbb Z}}

\def\cn{{\mathbb N}}

\def\L{{\mathcal{L}}}
\def\LV{{\mathscr{L}}}

\def\supp{{\mathop\mathrm{\,supp\,}}}
\def\dist{{\mathop\mathrm {\,dist\,}}}

\def\gz{{\gamma}}

\def\r{\right}
\def\lf{\left}

\def\eqref#1{(\ref{#1})}

\newtheorem{theorem}{Theorem}[section]

\newtheorem{corollary}[theorem]{Corollary}

\theoremstyle{definition}
\newtheorem{remark}[theorem]{Remark}

\newtheorem{thm}{Theorem}[section]
\newtheorem{lem}[thm]{Lemma}
\newtheorem{prop}[thm]{Proposition}

\newtheorem{rem}[thm]{Remark}

\numberwithin{equation}{section}

\begin{document}
\arraycolsep=1pt
\author{Renjin Jiang \& Fanghua Lin}
\title{\bf\Large Riesz transform on exterior Lipschitz domains and applications
\footnotetext{\hspace{-0.35cm}
2010 \emph{Mathematics Subject Classification}. 42B37, 35J05, 35J25.
\endgraf
{\it Key words and phrases}: Riesz transform, Dirichlet operators, Neumann operators,
exterior Lipschitz domain, harmonic function
\endgraf
}}

\maketitle

\begin{center}
\begin{minipage}{11.5cm}\small
{\noindent{\bf Abstract}. Let ${\mathscr{L}}=-\text{div}A\nabla$ be a uniformly elliptic operator on $\mathbb{R}^n$, $n\ge 2$.
Let $\Omega$ be an exterior Lipschitz domain, and let ${\mathscr{L}}_D$ and ${\mathscr{L}}_N$ be the operator ${\mathscr{L}}$ on $\Omega$ subject to the Dirichlet and Neumann boundary values, respectively. We establish the boundedness of the Riesz transforms $\nabla{\mathscr{L}}_D^{-1/2}$, $\nabla {\mathscr{L}}_N^{-1/2}$ in  $L^p$ spaces. As a byproduct, we show the reverse inequality
$\|{\mathscr{L}}_D^{1/2}f\|_{L^p(\Omega)}\le C\|\nabla f\|_{L^p(\Omega)}$ holds for any $1<p<\infty$.
The proof can be generalized to show  the boundedness of the Riesz transforms, for
operators with VMO coefficients on exterior Lipschitz or $C^1$ domains. The estimates can be also applied to the inhomogeneous Dirichlet and Neumann problems.  These results are new even for the Dirichlet and Neumann of the Laplacian operator on the exterior Lipschitz and $C^1$ domains.
}\end{minipage}
\end{center}
\vspace{0.2cm}

\vspace{0.2cm}
\tableofcontents

\section{Main results}
\hskip\parindent In this paper, we consider the Riesz transform for the Dirichlet and Neumann operators  on exterior Lipschitz domains. A domain  $\Omega\subset\rn$ is an exterior Lipschitz domain,
if the boundary of $\Omega$ is a finite union of parts of rotated graphs of Lipschitz functions,
$\Omega$ is connected and $\Omega^c=\rn\setminus\Omega$ is bounded. Suppose that $A(x)\in L^\infty(\rn)$  is a symmetric matrix that satisfies the uniformly elliptic condition, i.e.,
$$c|\xi|^2\le \langle A(x)\xi,\xi\rangle\le C|\xi|^2, \ \forall\ \xi\in\rn\ \& \, \forall \ x\in\rn. $$
In what follows, we denote by $\LV$ the operator $-\text{div}A\nabla$ on $\rn$,
and by $\LV_D$, $\LV_N$ the operator $-\text{div}A\nabla$ on $\Omega$ subject to the Dirichlet and Neumann boundary conditions, respectively. When $A=I_{n\times n}$, we simply denote them by $\Delta,\,\Delta_D,\,\Delta_N$, respectively.

{The study of Riesz transform was initiated by Riesz \cite{ri28} in 1928, where he proved via complex analysis the boundedness of the Hilbert transform (one dimension). 
  The extension to high dimensions was settled by Calder\'on-Zygmund \cite{cz52} in 1952, where the fundamental tool Calder\'on-Zygmund decomposition was developed. We refer the reader to
  \cite[Chapter 4]{gra08} for more details.}
  For bounded Lipschitz domains, the behavior of Riesz transform for the Dirichlet operators $\LV_D$ was solved by Shen \cite{shz05}, the case of Neumann operators follows from Auscher and Tchamitchian \cite{at01} and Geng \cite{ge12} (see Remark \ref{poincare} (ii) below).

   For $1\le p<\infty$, we denote by $W^{1,p}_0(\Omega)$, $W^{1,p}(\Omega)$ the completion of $C^{\infty}_c(\Omega)$, $C^\infty_c(\rn)$, respectively, under the norm
$\|f\|_{L^p(\Omega)}+\|\nabla f\|_{L^p(\Omega)}.$
The homogeneous Sobolev spaces  $\dot{W}^{1,p}_0(\Omega)$, $\dot{W}^{1,p}(\Omega)$, are the completion of
$C^{\infty}_c(\Omega)$, $C^\infty_c(\rn)$, respectively, under the quasi-norm
$\|\nabla f\|_{L^p(\Omega)}.$ Denote the H\"older conjugate of $p\ge 1$ by $p'$.
For $1<p<n$, let $p^\ast=\frac{np}{n-p}$, and for $p\ge \frac{n}{n-1}$,
let $p_\ast=\frac{np}{n+p}$.

{ An important application of the boundedness of the Riesz transforms is to show the equivalence of different defined Sobolev norms; see e.g. \cite{ac05,acdh,at98,at01,kvz16}.
For instance, $L^p$-boundedness of $\nabla \LV_D^{-1/2}$ implies that
 \begin{equation}\label{app-riesz}
 \|\nabla f \|_{L^p(\Omega)}\le C\|\LV_D^{1/2} f \|_{L^p(\Omega)},\quad\ \forall\, f\in \dot{W}^{1,p}_0(\Omega)
 \end{equation}
 and by duality that
 \begin{equation}\|\LV_D^{1/2} f \|_{L^{p'}(\Omega)}\le C\|\nabla f \|_{L^{p'}(\Omega)},\quad\ \forall\, f\in \dot{W}^{1,p'}_0(\Omega),
 \end{equation}
 where $1<p,p'<\infty$ satisfying $1/p+1/p'=1$. The same holds for $\nabla \LV_N^{-1/2}$ with $f\in \dot{W}^{1,p}(\Omega)$; see Theorem \ref{app-dirichlet}
 and Theorem \ref{app-neumann} below for details. }

Moreover, the boundedness of the Riesz transforms is closely related to {solvability and regularity problem of the following inhomogeneous Dirichlet/Neumann equations (see \eqref{dp} and \eqref{np} below for detailed descriptions)}
$$
\begin{cases}
\LV_D u=-\mathrm{div} f& \, {\text{in}}\, \Omega,\\
u=0 & \, {\text{on}}\, \partial\Omega,
\end{cases}\leqno(D_p)
$$
and
$$
\begin{cases}
\LV_N u=-\mathrm{div} f& \, {\text{in}}\, \Omega,\\
\nu\cdot A\nabla u=\nu\cdot f & \, {\text{on}}\, \partial\Omega,
\end{cases}\leqno(N_p)
$$
where $\nu$ denotes the outward unit normal. We shall address the question as
an application of main results in the last section. We refer to \cite{fmm98,gr85,jk95,lm68,zan00}
for studies of these problems for the Dirichlet and Neumann of Laplacian on bounded smooth or Lipschitz domains,
see the last section for more results on operators with discontinuous coefficients.

Much less was known for exterior Lipschitz domains. Previously, by studying weighted operators in the one dimension, Hassell and Sikora \cite{hs09} discovered that the Riesz transform of the Dirichlet Laplacian $\Delta_D$ on the exterior of the unit ball is {\em not} bounded on $L^p$ for $p>2$ if $n=2$, and $p\ge n$ if $n\ge 3$; see \cite[Theorem 1.1 \& Remark 5.8]{hs09}. Hassell and Sikora also conjectured that,
 for smooth exterior domains, the Riesz transform of the Dirichlet Laplacian $\Delta_D$ is bounded for $1<p<n$ if $n\ge 3$, and the Riesz transform of the Neumann Laplacian $\Delta_N$
is bounded for all $1<p<\infty$.

We remark that for both operators $\LV_D$ and  $\LV_N$ on $\rn$, $n\ge 2$, the Riesz transform is always bounded on $L^p(\Omega)$ for $1<p\le 2$. In fact,
by the maximum principle, the heat kernel $p^D_t(x,y)$ of $e^{-t\LV_D}$ is controlled by the heat kernel $p_t(x,y)$ of $e^{-t\LV}$,  i.e.,
\begin{equation}\label{dirichlet-gaussian}
0\le p^D_t(x,y)\le p_t(x,y)\le \frac{C(n)}{t^{n/2}}e^{-\frac{|x-y|^2}{ct}}.
\end{equation}
While for the Neumann heat kernel $p^N_t(x,y)$ of $e^{-t\LV_N}$, we note that exterior Lipschitz domains are uniform domains (deduced from Herron-Koskela \cite{hk91}, see Proposition \ref{inn-uniform} below), the results of Gyrya and Saloff-Coste \cite{gsa11} then implies that
\begin{equation}\label{neumann-gaussian}
\frac{\tilde C}{t^{n/2}}e^{-\frac{|x-y|^2}{\tilde ct}} \le p^N_t(x,y)\le \frac{C}{t^{n/2}}e^{-\frac{|x-y|^2}{ct}}.
\end{equation}
Thus the results of Sikora \cite{Si} (see also \cite{cd99}) implies that the Riesz transform for both operators is always bounded on $L^p(\Omega)$ for $1<p\le 2$.

Recently, Killip, Visan and Zhang \cite{kvz16} established that for domains outside a smooth convex obstacle,  the Riesz transform $\nabla \Delta_D^{-1/2}$ is bounded for $1<p<n$, $n\ge 3$. Actually the results of
\cite{kvz16} also include the fractional cases, which we will not pursue in the present paper.

Our main results give a characterization of the boundedness of the Riesz transform on exterior Lipschitz domains.
\begin{theorem}[Dirichlet Operator]\label{main-dirichlet}
Let $\Omega\subset \rn$ be an exterior Lipschitz domain, $n\ge 3$.
Let $p\in (2,n)$. Then the followings are equivalent.

(i) The Riesz operator $\nabla\mathscr{L}_D^{-1/2}$ is bounded on $L^p(\Omega)$.

(ii)  There exist $C>0$ and $1<\alpha_1<\alpha_2<\infty$ such that for any ball $B(x_0,r)$ satisfying
$B(x_0,\alpha_2 r)\subset \Omega$ or $B(x_0,\alpha_2 r)\cap\partial\Omega\neq \emptyset$ with $x_0\in\partial \Omega$, and any weak solution $u$ of $\mathscr{L}_Du=0$ in $\Omega\cap B(x_0,\alpha_2r)$, satisfying additionally $u=0$ on $B(x_0,\alpha_2r)\cap\partial\Omega$ if $x_0\in\partial \Omega$, it holds
$$\lf(\fint_{B(x_0,r)\cap\Omega}|\nabla u|^{p}\,dx\r)^{1/p}\le
\frac{C}{r}\fint_{B(x_0,\alpha_1 r)\cap\Omega}|u|\,dx. \leqno({RH}_p)$$
\end{theorem}
We did not include the planar case, since, as we discussed above,  the results of \cite{hs09} imply that for the planar case, $\nabla\mathscr{L}_D^{-1/2}$ is not bounded for any $p>2$.
{Moreover, even for the Laplace operator,  $\nabla\Delta_D^{-1/2}$ is not $L^n$-bounded if $n\ge 3$ and hence not $L^p$-bounded for any $p> n$; see
\cite{hs09,kvz16} and Remark \ref{unbound-dirichlet} below.}

\begin{theorem}[Neumann Operator]\label{main-neumann}
Let $\Omega\subset \rn$ be an exterior Lipschitz domain, $n\ge 2$.
Let $p\in (2,\infty)$. Then the followings are equivalent.

(i) The Riesz operator $\nabla\mathscr{L}_N^{-1/2}$ is bounded on $L^p(\Omega)$.

 (ii) There exist $C>0$ and $1<\alpha_1<\alpha_2<\infty$ such that for any ball $B(x_0,r)$ with $x_0\in {\Omega}$ and any weak solution $u$ of $\mathscr{L}_Nu=0$ in $\Omega\cap B(x_0,\alpha_2r)$, satisfying additionally $\partial_\nu u=0$ on $B(x_0,\alpha_2r)\cap\partial\Omega$ if the set is not empty, it holds
$$\lf(\fint_{B(x_0,r)\cap \Omega}|\nabla u|^{p}\,dx\r)^{1/p}\le
\frac{C}{r}\fint_{B(x_0,\alpha_1 r)\cap \Omega}|u|\,dx. \leqno({RH}_p)$$
\end{theorem}

Let us explain a bit about  the proof. For the Neumann case, since $(\Omega,\LV_N)$ is stochastic complete, i.e., $e^{-t\LV_N}1=1$, the result follows from \cite[Theorem 1.9]{cjks16} (see also \cite{acdh}),
 provided that the Neumann heat kernel satisfies a two side Gaussian bounds. As we discussed previously, this
 follows from showing that $\Omega$ is inner uniform in the sense of \cite{gsa11}; see Proposition \ref{inn-uniform} below.

The proof for the Dirichlet case is much more involved. Since $(\Omega,\LV_D)$ is not stochastic complete,
previous results from \cite{acdh,cch06,cjks16,ji20,shz05} no longer work in this setting.
We shall incorporate some ideas from \cite{ji20} to give a new criteria for boundedness of singular
integral operators, see Theorem \ref{criteria-Riesz} below, and then by using the following
reverse inequality to complete proof.

\begin{theorem}[Reverse inequality]\label{reverse-ineq}
Let $\Omega\subset \rn$ be an exterior Lipschitz domain, $n\ge 2$. Then for any Dirichlet operator $\LV_D$
 and $1<p<\infty$, there exists $C>0$ such that for all $f\in \dot{W}^{1,p}_0(\Omega)$ it holds
$$\|\mathscr{L}_D^{1/2}f\|_{L^p(\Omega)}\le C\|\nabla f\|_{L^p(\Omega)}.$$
\end{theorem}

It is worth mentioning that the above inequality confirms a conjecture of Auscher-Tchamitchian
\cite[Remark 12]{at01} for the Dirichlet operators. Recall that \cite{at01} established the above
inequality for both Dirichlet and Neumann operators on bounded Lipschitz domains as well as special
Lipschitz domains (the open set above a Lipschitz graph), and global case was proved in \cite{at98}.
{For the Neumann operator, from the boundedness of the Riesz transform and a duality argument,
we see that the reverse inequality
$$\|\mathscr{L}_N^{1/2}f\|_{L^p(\Omega)}\le C\|\nabla f\|_{L^p(\Omega)}, \ \forall \ f\in \dot{W}^{1,p}(\Omega),$$
is true for $p\in (2-\delta,\infty)$ for some $\delta>0$ (cf. Theorem \ref{app-neumann} below).
One naturally expects that the reverse inequality holds 
true for all $1<p<\infty$ on any exterior Lipschitz domains. It does not seem accessible by our current methods. 
Recently,  
with some extra techniques, the problem has been settled in \cite{jy24}.}
{In this paper, we can show that  this reverse inequality holds for all $1<p<\infty$ on exterior $C^1$ domains, if the coefficients are in $VMO$ space
and have some growth control at infinity (cf. Corollary \ref{app-inf-neumann} below)}. Here and in what follows, we assume $A$ has VMO coefficients, i.e., $A\in VMO(\rn)$:
$$\lim_{r\to 0}\sup_{x\in \rn}\fint_{B(x,r)}|A(y)-A_{B(x,r)}|\,dy=0,$$
where $A_{B(x,r)}$ denotes the integral average of $A$ over $B(x,r)$.

Let us apply these characterizations to some concrete cases. For the case of bounded
domains, the behavior of the Riesz transform essentially depends on the geometry of the boundary and local regularity (small scale) of harmonic functions; see \cite{aq02,ge12,shz05} for instance. For the exterior case, apart from the above two properties, one has to use a large scale regularity of harmonic functions.
In fact,
for any given $p>2$, by \cite[Proposition 1.1]{jl20}, there exists $A_0$, which is sufficiently smooth (hence is in the VMO space), however, for the balls $\{B(0,2^j)\}_{j\to\cn}$ there exist harmonic functions $u_j$ on
$B(0,2^j)$ such that $(RH_p)$ does not hold for $u_j$ on
$B(0,2^j)$ as $j\to \infty$. In the exterior case, we consider a sequence of interior balls $B(x_j,2^j)\subset\Omega$ such that each ball is far from others, let $A(x)=A_0(x-x_j)$ on each $B(x_j,2^j)$ and extend $A$ smoothly to $\rn\setminus \cup_jB(x_j,r_j)$. Then $(RH_p)$ fails on $B(x_j,2^j)$ as $j\to\infty$, and the operators $\nabla \LV_D^{-1/2}$ and $\nabla \LV_N^{-1/2}$
are not bounded on $L^p(\Omega)$.
Therefore some more condition is needed to guarantee the large scale behavior of harmonic functions.  We define
$$p_\LV:=\sup\{p>2: \,\nabla \LV^{-1/2}\ \text{is bounded on}\, L^p(\rn)\}.$$
According to \cite{ac05,at98}, $p_\LV\in (2,\infty]$. Moreover, Kenig's example (cf. \cite{at98,shz05}) shows
that $p_\LV$ can be arbitrarily close to 2; see also \cite[Proposition 1.1]{jl20}. For $\LV=\Delta$ being the Laplacian, one has $p_\LV=\infty$.

\begin{theorem}\label{app-dirichlet}
 Let $\Omega\subset \rn$ be an exterior Lipschitz domain, $n\ge 3$.
 Suppose that  $A\in VMO(\rn)$.

 (i) There exist  $\epsilon>0$ and  $C>1$  such that for all $f\in \dot{W}^{1,p}_0(\Omega)$ it holds
$$C^{-1}\|\nabla f\|_{L^p(\Omega)}\le \|\LV_D^{1/2}f\|_{L^p(\Omega)}\le C\|\nabla f\|_{L^p(\Omega)},$$
where $1<p<\min\{n,p_\LV,3+\epsilon\}$.

(ii) If $\Omega$ is $C^1$, then the conclusion of $(i)$ holds for all $1<p<\min\{n,p_\LV\}$.
\end{theorem}

\begin{theorem}\label{app-neumann}
Let $\Omega\subset \rn$ be an exterior Lipschitz domain, $n\ge 2$.
Let $A\in VMO(\rn)$.

(i) There exist $\epsilon>0$ and $C>1$  such that for all $f\in L^p(\Omega)$ it holds
$$\|\nabla \LV_N^{-1/2}f\|_{L^p(\Omega)}\le C \|f\|_{L^p(\Omega)},$$
where $1<p<\min\{p_\LV,3+\epsilon\}$ when $n\ge 3$, $ 1<p<\min\{p_\LV,4+\epsilon\}$  when $n=2$.

Moreover, it holds  for all $f\in \dot{W}^{1,p}(\Omega)$ that
$$C^{-1}\|\nabla f\|_{L^p(\Omega)}\le \|\LV_N^{1/2}f\|_{L^p(\Omega)}\le C\|\nabla f\|_{L^p(\Omega)},$$
where $\max\{p_\LV ', (3+\epsilon)'\}<p<\min\{p_\LV,3+\epsilon\}$ when $n\ge 3$, $ \max\{p_\LV ', (4+\epsilon)'\}<p<\min\{p_\LV,4+\epsilon\}$  when $n=2$.

(ii) If $\Omega$ is $C^1$, then the conclusion of $(i)$ holds for $\epsilon=\infty$.
\end{theorem}
In the above two results, the index $3+\epsilon$ when $n\ge 3$, $4+\epsilon$  when $n=2$, coming from the
effect of Lipschitz boundary, is sharp already for Dirichlet and Neumann Laplacians $\Delta_D, \Delta_N$;
see \cite{jk95,mm01}. The restriction of $p<p_\LV$ is also necessary as previously explained.
For the particular case $\LV=\Delta$, the above two Theorems confirm that the results conjectured in \cite{hs09} hold on $C^1$ domains.
 { Note that to ensure the $L^p$-boundednss of $\nabla \LV_N^{-1/2}$ for all $p\in (1,\infty)$,
 the smoothness condition $C^1$ cannot be weakened to Lipschitz continuity (see \cite[Section 12]{fmm98})}.

From our previous work \cite{jl20} that, we know that one can take $p_\LV=\infty$ if  $A\in VMO(\rn)$ satisfies
$$\fint_{B(x_0,r)}|A-I_{n\times n}|\,dx\le \frac{C}{r^\delta}$$
for some $\delta>0$, all $r>1$ and all $x_0\in\rn$.
We shall include this case in the last section, and discuss applications to the inhomogeneous Dirichlet/Neumann problem $(D_p)$ and $(N_p)$ there.

The paper is organized as follows. In Section 2, we provide the proof of the reverse  inequality,
Theorem \ref{reverse-ineq}. In Section 3, we provide the proof for the Dirichlet operators,
and prove Theorem \ref{main-dirichlet} and Theorem \ref{app-dirichlet}.
In Section 4, we treat the Neumann case, and prove  Theorem \ref{main-neumann} and Theorem \ref{app-neumann}. In the last section, we shall provide some more detailed examples and
applications to $(D_p)$ and $(N_p)$.

Throughout the work, we denote by $C,c$ positive constants which are independent of the
main parameters, but which may vary from line to line. We sometimes use $a \lesssim b$ to mean that
$a\le C b$, and $a\sim b$ to mean that $ca\le b\le Cb$. Throughout the paper,  $\Omega$
is an exterior Lipschitz domain unless otherwise specified. Up to a translation,  we may and do assume that the origin belongs
to the interior of $\rn\setminus\Omega$ for simplicity of notions.

\section{Reverse inequality for the Dirichlet operator}
\hskip\parindent  In this section, we provide the proof for Theorem \ref{reverse-ineq}.
The main approach combines some results of \cite{at01} and \cite{kvz16}, and depends on a comparison
result for the difference of the heat kernels on  $\rn$ and $\Omega$.

Recall that $p_t^D(x,y)$, $p^N_t(x,y)$ and $p_t(x,y)$ denote the heat kernels of
$P_t^D=e^{-t\LV_D}$, $P^N_t=e^{-t\LV_N}$ and $P_t=e^{-t\LV}$, respectively.
Recall also that it follows from the maximal principle that for all $x,y\in \overline{\Omega}$ and $t>0$
\begin{equation}\label{dirichlet-gaussian-1}
0\le p^D_t(x,y)\le p_t(x,y)\le \frac{C(n)}{t^{n/2}}e^{-\frac{|x-y|^2}{ct}}.
\end{equation}

The following Littlewood-Paley equivalence is a special case of \cite[Theorem 4.3]{kvz16},
 i.e., by taking $s=1$ and $k=1$ there. Note that although the main result of \cite{kvz16} focuses on
 an exterior domain outside a smooth convex obstacle, \cite[Theorem 3.1 \& Theorem 4.3]{kvz16}
 works however on general domains, as only a Gaussian upper bound for the heat kernel is needed 
 (see \cite[Theorem 3.1]{dos02} for multiplier theorem in abstract setting);  see also \cite[\S IV.5.3]{ste70}.

\begin{theorem}\label{littlewood-paley}
(i) It holds for any $f\in C^\infty_c(\rn)$ and $1<p<\infty$ that
\begin{eqnarray*}
\|\LV^{1/2}f\|_{L^p(\rn)}\sim \left\|\left(\sum_{j\in \zz} 2^{-2j} \left|\left(P_{2^{2j}}-P_{2^{2j+2}}\right) f\right|^2\right)^{1/2}\right\|_{L^p(\rn)}.
\end{eqnarray*}

(ii) For any $g\in C^\infty_c(\Omega)$ and $1<p<\infty$, it holds
\begin{eqnarray*}
\|\LV_D^{1/2}g\|_{L^p(\Omega)}\sim \left\|\left(\sum_{j\in \zz} 2^{-2j} \left|\left(P_{2^{2j}}^D-P_{2^{2j+2}}^D\right) g\right|^2\right)^{1/2}\right\|_{L^p(\Omega)}.
\end{eqnarray*}
\end{theorem}

A key observation is the following upper bound for the difference between heat kernels on space and the domain $\Omega$. Recall that we always assume that the origin belongs
to the interior of $\rn\setminus\Omega$.
\begin{prop}\label{heat-comparison}
Let $\Omega\subset \rn$ be an exterior Lipschitz domain, $n\ge 2$. Then there exist $c,C,R,\delta>1$ such that
$\rn\setminus \Omega\subset B(0,R)$, and for all $x,y\in \rn\setminus B(0,\delta R)$ and $t>0$,
it holds that
\begin{equation}\label{upper-difference}
0\le p_t(x,y)-p^D_t(x,y)\le  Ct^{-n/2}e^{-\frac{|x-y|^2+|x|^2+|y|^2}{ct}}.
\end{equation}
\end{prop}
\begin{proof} Choose a large enough $R$ such that $\rn\setminus \Omega\subset B_{R}=B(0,R)$. Then
$ \rn\setminus B_R=B_R^c\subset \Omega$.

Denote by $\LV_R$ the operator induced by $\LV$ on the domain $\{x\in \rn:|x|>R\}$,
subject to the Dirichlet boundary condition, and denote by $p^R_t(x,y)$ the heat kernel
of the heat semigroup $e^{-t\LV_R}$. Since $B_R^c\subset\Omega$, it follows from the maximal principle that for all $x,y\in\rn$,
\begin{equation}\label{upper-difference-1}
0\le p_t(x,y)-p^\Omega_t(x,y)\le p_t(x,y)-p^R_t(x,y)\le Ct^{-n/2}e^{-\frac{|x-y|^2}{ct}}.
\end{equation}
Let $\delta,\delta_1>1$ be two constants to be fixed later. For any non-negative function $f\in L^1(\rn)$
supported in $B_{\delta R}^c$ with $\|f\|_{L^1(\rn)}=1$.
Consider the function
\begin{equation}
  u(x,t):=\int_{B_{\delta R}^c} \left[\delta_1p_t(x,0)-p_t(x,y)+p^{R}_t(x,y)\right]f(y)\,dy,
\end{equation}
where $x\in B_R^c$. Then $u$ is a solution to the heat equation
$$(\partial_t+\LV)u=0$$
on $B_R^c\times (0,\infty)$.

Note that, for any $x$ with $|x|>R$, it holds
$$u(x,0)=\lim_{t\to 0} \delta_1p_t(x,0) - f(x)+f(x)=0.$$
When $|x|=R$, we have $p^{R}_t(x,y)=0$, and
$$p_t(x,0)\ge Ct^{-n/2}e^{-\frac{R^2}{ct}}.$$
Thus we can choose $\delta,\delta_1>1$ such that for any $y\in B_{\delta R}^c$, $|x|=R$ and $t>0$,
$$p_t(x,y)\le Ct^{-n/2}e^{-\frac{|x-y|^2}{ct}}\le Ct^{-n/2}e^{-\frac{\delta^2R^2}{ct}}\le \delta_1p_t(x,0).$$
Therefore, for all $x$ with $|x|=R$ and $t>0$, it holds
$$u(x,t)\ge 0.$$
The maximal principle then implies that for any $x\in B_R^c$ and $t>0$
$$u(x,t)\ge 0,$$
for any  non-negative function $f\in L^1(\rn)$
supported in $B_{\delta R}^c$ with $\|f\|_{L^1(\rn)}=1$.
We therefore, deduce that, for any $x\in B_R^c$, $y\in B_{\delta R}^c$  and $t>0$ that
\begin{equation}
0\le p_t(x,y)-p^{R}_t(x,y)\le \delta_1p_t(x,0).
\end{equation}

In particular, for all $x,y\in B_{\delta R}^c$  and $t>0$, it holds that
\begin{equation}
0\le p_t(x,y)-p^{R}_t(x,y)\le \delta_1p_t(x,0)\le Ct^{-n/2}e^{-\frac{|x|^2}{ct}}.
\end{equation}
The symmetry of heat kernel implies that
\begin{equation}
0\le p_t(x,y)-p^{R}_t(x,y)\le \delta_1p_t(0,y)\le Ct^{-n/2}e^{-\frac{|y|^2}{ct}}.
\end{equation}
The two inequalities together with \eqref{upper-difference} imply that for all $x,y\in B_{\delta R}^c$  and $t>0$,
\begin{eqnarray*}
0\le p_t(x,y)-p^{\Omega}_t(x,y)&&\le p_t(x,y)-p^{R}_t(x,y)\\
&&\le \min\left\{Ct^{-n/2}e^{-\frac{|x|^2}{ct}},\, Ct^{-n/2}e^{-\frac{|y|^2}{ct}},\, Ct^{-n/2}e^{-\frac{|x-y|^2}{ct}}\right\}\\
&&\le Ct^{-n/2}e^{-\frac{|x-y|^2+|x|^2+|y|^2}{ct}},
\end{eqnarray*}
which completes the proof.
\end{proof}
\begin{rem} (i) In the case of $\mathscr{L}=\Delta$ one can give a detailed calculation of constants in \eqref{upper-difference}; 
see \cite[p. 5911]{kvz16} for instance. 

Let us suppose that $\Omega^c\subset B(0,R)$,  $\mathscr{L}=\Delta$ on $\mathbb{R}^n$, $n\ge 2$.
For $y\in \Omega$, consider the hyperplane $\mathbb{H}_y$ that is tangential to $\partial B(0,R)$ such that
$\mathrm{dist}(y,B(0,R))=\mathrm{dist}(y, \mathbb{H}_y).$ 
Let $H$ be the half space that contains $y$, and $p^H_t(x,z)$ be the Dirichlet heat kernel on $H$.
Then it holds for $x\in H$ that
$$p^H_t(x,y)=p_t(x,y)-p_t(x,y')=\frac{1}{(4\pi t)^{n/2}}e^{-\frac{|x-y|^2}{4t}}- \frac{1}{(4\pi t)^{n/2}}e^{-\frac{|x-y'|^2}{4t}},$$
where $y'$ is the reflection of $y$ w.r.t. $\mathbb{H}_y$. Moreover, for $x\notin H$,
$p^H_t(x,y)=0.$
Thus we see that
\begin{eqnarray*}
p_t(x,y)-p^D_t(x,y)&&\le p_t(x,y)-p^H_t(x,y)
=\begin{cases}\frac{1}{(4\pi t)^{n/2}}e^{-\frac{|x-y'|^2}{4t}}, &x\in H\\
\frac{1}{(4\pi t)^{n/2}}e^{-\frac{|x-y|^2}{4t}}, & x\notin H.
\end{cases}
\end{eqnarray*}
Suppose now $x,y\in \mathbb{R}^n\setminus \overline{B(0,2R)}$. If $x\in H$, then  we have
\begin{eqnarray*}
|x-y'|^2=\mathrm{dist}(x,\ell_{yy'})^2+\left(|y|-R+\mathrm{dist}(x,\mathbb{H}_y)\right)^2\ge \max\left\{|x-y|^2,\frac{|y|^2}{4},|x|^2\right\},
\end{eqnarray*}
where $\ell_{yy'}$ denotes the line passing through $y,y'$. If $x\notin H$, then
\begin{eqnarray*}
|x-y|^2=\mathrm{dist}(x,\ell_{yy'})^2+\left(|y|-R+\mathrm{dist}(x,\mathbb{H}_y)\right)^2\ge \max\left\{|x-y|^2,\frac{|y|^2}{4},|x|^2\right\}.
\end{eqnarray*}
The above three inequalities imply that for $x,y\in \mathbb{R}^n\setminus \overline{B(0,2R)}$, $n\ge 2$, it holds
\begin{eqnarray*}
p_t(x,y)-p^D_t(x,y)&\le \frac{1}{(4\pi t)^{n/2}}e^{-\frac{|x-y|^2+|x|^2+|y|^2}{24t}},
\end{eqnarray*}
 which agrees with \eqref{upper-difference}. 

(ii) Let us compare \eqref{upper-difference} with  \cite[Example 1.3]{gsc02}, which  states that for large enough $x,y$ and all $t>0$, 
\begin{equation}\label{heat}
p_t^D(x,y)\sim \frac{\log|x|\log|y| }{t(\log(1+\sqrt t)+\log|x|)(\log(1+\sqrt t)+\log|y|)}e^{-\frac{|x-y|^2}{t}}.
\end{equation}
This however does not contradict with \eqref{upper-difference-1}. 
In fact, \eqref{upper-difference} is trivial if $|x|^2+|y|^2\le ct$ or $|x|^2+|y|^2\le 100|x-y|^2$, \eqref{dirichlet-gaussian} 
gives the estimate in this case. 
Let us suppose that  
$$|x|^2+|y|^2>> t \  \&\  |x|^2+|y|^2> 100|x-y|^2.$$
Suppose  further that $|x|\ge |y|$. Then 
$$|y|\ge |x|-|x-y|\ge \frac{9}{10}|x|-\frac{|y|}{10},$$
and $|y|\ge 9|x|/11$ is large also. So as $|x|^2+|y|^2>>t$, both $x,y$ are large enough comparing to $\sqrt t$. In this case, 
\eqref{heat} reduces to 
\begin{equation*}
p_t^D(x,y)\sim \frac 1t e^{-\frac{|x-y|^2}{t}}.
\end{equation*}
\end{rem}

We can now give the proof of the reverse inequality.
The main approach follows \cite{kvz16}, where we use the previous proposition and
some results on the reverse inequality on bounded domains and the space $\rn$ from \cite{at98,at01}.
\begin{proof}[Proof of Theorem \ref{reverse-ineq}] Since the Riesz transform $\nabla \LV_D^{-1/2}$ is bounded on $L^p(\Omega)$ for $1<p\le 2$ (cf. \cite{Si}), the reverse inequality
$$\|\mathscr{L}_D^{1/2}f\|_{L^{q}(\Omega)}\le C\|\nabla f\|_{L^{q}(\Omega)}$$
for all $2\le q<\infty$ follows from duality.

Thus we only to prove the reverse inequality for $1<p<2$.
We choose a bump function $\phi\in C^\infty(\rn)$ such that $1-\phi\in C^\infty_c(\rn)$ with $\supp (1-\phi)\subset B(0,3\delta R)$ and $1-\phi=1$ on $B(0,2\delta R)$.
For any $g\in C^\infty_c(\Omega)$, by Theorem \ref{littlewood-paley}, we have
\begin{eqnarray}\label{est-large}
\|\LV_D^{1/2}(g\phi)\|_{L^p(\Omega)}&&\sim \left\|\left(\sum_{j\in \zz} 2^{-2j} \left|\left(P_{2^{2j}}^D-P_{2^{2j+2}}^D\right) (g\phi)\right|^2\right)^{1/2}\right\|_{L^p(\Omega)}\nonumber\\
&&\le C\left\|\left(\sum_{j\in \zz} 2^{-2j} \left|\left(P_{2^{2j}}-P_{2^{2j+2}}-P_{2^{2j}}^D+P_{2^{2j+2}}^D\right) (g\phi)\right|^2\right)^{1/2}\right\|_{L^p(\Omega)}\nonumber\\
&&\quad+C\left\|\left(\sum_{j\in \zz} 2^{-2j} \left|\left(P_{2^{2j}}-P_{2^{2j+2}}\right) (g\phi)\right|^2\right)^{1/2}\right\|_{L^p(\Omega)}\nonumber\\
&&\le C\left\|\sum_{j\in \zz} 2^{-j} \left|\left(P_{2^{2j}}-P_{2^{2j+2}}-P_{2^{2j}}^D+P_{2^{2j+2}}^D\right) (g\phi)\right|\right\|_{L^p(\Omega)}+C\|\LV^{1/2}(g\phi)\|_{L^p(\rn)}.
\end{eqnarray}
By the fact $\|\LV f\|_{L^p(\rn)}\lesssim \|\nabla f\|_{L^p(\rn)}$ from Auscher-Tchamitchian \cite[Chapter 4, Proposition 19]{at98} and the Sobolev inequality, we deduce that
\begin{eqnarray*}
\|\LV^{1/2}(g\phi)\|_{L^p(\rn)}&&\le C\|\nabla (g\phi)\|_{L^p(\rn)}\le C\|\nabla g\|_{L^p(\rn)}+C\|g\|_{L^p(B_{3\delta R})}\\
&&\le C\|\nabla g\|_{L^p(\Omega)}+\|g\|_{L^{p^\ast}(B_{3\delta R})}\\
&&\le C\|\nabla g\|_{L^p(\Omega)}.
\end{eqnarray*}
To deal with the remaining term in \eqref{est-large}, by applying  Proposition \ref{heat-comparison},
we have that for $x,y\in \rn\setminus B_{\delta R}$,
\begin{eqnarray*}
0\le p_{t}(x,y)-p_t^D(x,y)\le Ct^{-n/2}e^{-\frac{|x|^2+|y|^2+|x-y|^2}{ct}},
\end{eqnarray*}
which implies that
\begin{eqnarray*}
&&\sum_{j\in \zz} 2^{-j} \left|\left(P_{2^{2j}}-P_{2^{2j+2}}-P_{2^{2j}}^D+P_{2^{2j+2}}^D\right) (g\phi)(x)\right|\\
&&\le C\sum_{j\in \zz} 2^{-j}\int_{B_{2\delta R}^c} 2^{-jn}e^{-c2^{-2j}(|x|^2+|y|^2+|x-y|^2)}|g\phi(y)|\,dy\\
&&\le C\sum_{j\in \zz,\,2^{2j}\le |x|^2+|y|^2+|x-y|^2 } 2^{-j}\int_{B_{2\delta R}^c} 2^{-jn}e^{-c2^{-2j}(|x|^2+|y|^2+|x-y|^2)}|g\phi(y)|\,dy\\
&&\quad +C\sum_{j\in \zz,\,2^{2j}> |x|^2+|y|^2+|x-y|^2 } 2^{-j}\int_{B_{2\delta R}^c} 2^{-jn}e^{-c2^{-2j}(|x|^2+|y|^2+|x-y|^2)}|g\phi(y)|\,dy\\
&&\le C\int_{B_{2\delta R}^c} \frac{|g\phi(y)|}{(|x|+|y|+|x-y|)^{n+1}}\,dy.
\end{eqnarray*}
For $x\in B_{\delta R}\cap \Omega$ and $y\in \rn\setminus B_{2\delta R}$, by noting that
$$|x|+|y|\le |x-y|+2|x|\le c|x-y|\le c(|x|+|y|),$$
we deduce from the upper Gaussian bounds of the heat kernel \eqref{dirichlet-gaussian} that
\begin{eqnarray*}
&&\sum_{j\in \zz} 2^{-j} \left|\left(P_{2^{2j}}-P_{2^{2j+2}}-P_{2^{2j}}^D+P_{2^{2j+2}}^D\right) (g\phi)(x)\right|\\
&&\le C\sum_{j\in \zz} 2^{-j}\int_{B_{2\delta R}^c} 2^{-jn}e^{-c2^{-2j}(|x-y|^2)}|g\phi(y)|\,dy\\
&&\le C\sum_{j\in \zz} 2^{-j}\int_{B_{2\delta R}^c} 2^{-jn}e^{-c2^{-2j}(|x|^2+|y|^2+|x-y|^2)}|g\phi(y)|\,dy\\
&&\le C\int_{B_{2\delta R}^c} \frac{|g\phi(y)|}{(|x|+|y|+|x-y|)^{n+1}}\,dy.
\end{eqnarray*}
We therefore conclude that
\begin{eqnarray*}
&& \left\|\sum_{j\in \zz} 2^{-j} \left|\left(P_{2^{2j}}-P_{2^{2j+2}}-P_{2^{2j}}^D+P_{2^{2j+2}}^D\right) (g\phi)\right|\right\|_{L^p(\Omega)}\nonumber\\
&&\le C\left\|\int_{B_{2\delta R}^c} \frac{|g\phi(y)|}{(|x|+|y|+|x-y|)^{n+1}}\,dy\right\|_{L^p(\Omega)}\nonumber\\
&&\le C\sup_{h:\,\|h\|_{L^{p'}(\Omega)}\le 1} \int_{\Omega}\int_{B_{2\delta R}^c} \frac{|g\phi(y)|h(x)}{(|x|+|y|+|x-y|)^{n+1}}\,dy\,dx.\nonumber
\end{eqnarray*}
Similar to \cite[Lemma 5.1]{kvz16}, for $\alpha>0$ such that $\max\{\alpha p,\alpha p'\}<n$, by applying the H\"older inequality, we have
\begin{eqnarray*}
\int_{\Omega}\int_{B_{2\delta R}^c} \frac{|g\phi(y)|h(x)}{(|x|+|y|+|x-y|)^{n+1}}\,dy\,dx
&&\le C\left( \int_{\Omega}\int_{B_{2\delta R}^c} \frac{|g\phi(y)|^p}{|y|^p}
\frac{|y|^{\alpha p}}{|x|^{\alpha p}}\frac{|y|}{(|x|+|y|+|x-y|)^{n+1}}\,dy\,dx\right)^{1/p}\nonumber\\
&&\quad\times\left( \int_{\Omega}\int_{B_{2\delta R}^c} \frac{|x|^{\alpha p'}}{|y|^{\alpha p'}}\frac{|h(x)|^{p'}|y|}{(|x|+|y|+|x-y|)^{n+1}}\,dy\,dx\right)^{1/p'},\nonumber\\
\end{eqnarray*}
where
\begin{eqnarray*}
\int_{\Omega}
\frac{|y|^{\alpha p}}{|x|^{\alpha p}}\frac{|y|}{(|x|+|y|+|x-y|)^{n+1}}\,dx&&\le
\int_{\{x:\,|x|\le 2|y|\}}\cdots\,dx+\int_{\{x:\,|x|>2|y|\}}\cdots\,dx\\
&&\le C|y|^{\alpha p+1-n-1} \int_{\{x:\,|x|\le 2|y|\}}|x|^{-\alpha p}\,dx
+C|y| \int_{\{x:\,|x|>2|y|\}}|x|^{-n-1}\,dx\\
&&\le C,
\end{eqnarray*}
and similarly,
\begin{eqnarray*}
\int_{B_{2\delta R}^c} \frac{|x|^{\alpha p'}}{|y|^{\alpha p'}}\frac{|y|}{(|x|+|y|+|x-y|)^{n+1}}\,dy
&&\le \int_{\{y:\,|y|\le 2|x|\}}\cdots\,dy+\int_{\{y:\,|y|>2|x|\}}\cdots\,dy\\
&&\le C|x|^{\alpha p'+1-n-1}\int_{\{y:\,|y|\le 2|x|\}}|y|^{-\alpha p'}\,dy+C|x|^{\alpha p'}\int_{\{y:\,|y|>2|x|\}}|y|^{-n-\alpha p'}\,dy\\
&&\le C.
\end{eqnarray*}
These two estimates imply that
\begin{eqnarray*}
\int_{\Omega}\int_{B_{2\delta R}^c} \frac{|g\phi(y)|h(x)}{(|x|+|y|+|x-y|)^{n+1}}\,dy\,dx
&&\le C\left(\int_{B_{2\delta R}^c} \frac{|g\phi(y)|^p}{|y|^p}\,dy\right)^{1/p}
\times\left( \int_{\Omega}|h(x)|^{p'}\,dx\right)^{1/p'},
\end{eqnarray*}
and hence,
\begin{eqnarray*}
&& \left\|\sum_{j\in \zz} 2^{-j} \left|\left(P_{2^{2j}}-P_{2^{2j+2}}-P_{2^{2j}}^D+P_{2^{2j+2}}^D\right) (g\phi)\right|\right\|_{L^p(\Omega)}\nonumber\\
&&\sim \sup_{h:\,\|h\|_{L^{p'}(\Omega)}\le 1} \int_{\Omega}\int_{B_{2\delta R}^c} \frac{|g\phi(y)|h(x)}{(|x|+|y|+|x-y|)^{n+1}}\,dy\,dx\nonumber\\
&&\le C\left(\int_{B_{2\delta R}^c} \frac{|g\phi(y)|^p}{|y|^p}\,dy\right)^{1/p}.
\end{eqnarray*}
Since $1<p<2$ and $g\in C^\infty_c(\Omega)$, Hardy's inequality (cf. \cite[Theorem]{ckn84}) implies that
\begin{eqnarray}\label{est-large-1}
\|\LV_D^{1/2}(g\phi)\|_{L^p(\Omega)}&&\le C\left\|\sum_{j\in \zz} 2^{-j} \left|\left(P_{2^{2j}}-P_{2^{2j+2}}-P_{2^{2j}}^D+P_{2^{2j+2}}^D\right) (g\phi)\right|\right\|_{L^p(\Omega)}+C\|\LV^{1/2}(g\phi)\|_{L^p(\rn)}\nonumber\\
&&\le C\left(\int_{B_{2\delta R}^c} \frac{|g\phi(y)|^p}{|y|^p}\,dy\right)^{1/p}+C\|\nabla g\|_{L^p(\Omega)}\nonumber\\
&&\le C\left(\int_{B_{2\delta R}^c} \frac{|g(y)|^p}{|y|^p}\,dy\right)^{1/p}+C\|\nabla g\|_{L^p(\Omega)}\nonumber\\
&&\le C\|\nabla g\|_{L^p(\Omega)}.
\end{eqnarray}

For the remaining term $g(1-\phi)$,
by \cite[Thoerem 1]{at01}, it holds that
\begin{eqnarray*}
\|\LV^{1/2}_D(g(1-\phi))\|_{L^p(\Omega)}&&\le C\|\nabla(g(1-\phi))\|_{L^p(\Omega)}+ C\|g(1-\phi)\|_{L^p(\Omega)}\\
&&\le C\|\nabla g\|_{L^p(\Omega)}+ C\|g\|_{L^p(B_{3\delta R})}\\
&&\le C\|\nabla g\|_{L^p(\Omega)}+ C\|g\|_{L^{p^\ast}(B_{3\delta R})}\\
&&\le C\|\nabla g\|_{L^p(\Omega)},
\end{eqnarray*}
where the last inequality follows from the Sobolev inequality (recall that here $1<p<2$).

The last two estimates complete the proof.
\end{proof}

\begin{rem}
In \cite[Theorem 1.3]{kvz16}, the proof  
depends essentially on the heat kernel estimate
deduced by Zhang \cite[Theorem 1.1]{zha03}, which states that  for exterior $C^{1,1}$ domains $\Omega$ in $\mathbb{R}^n$, $n\ge 3$, enjoys an estimate as 
$$p_t^D(x,y)\le Ct^{-n/2}\left(\frac{\mbox{dist}(x,\Omega^c)}{\sqrt t\wedge\mbox{diam}(\Omega^c)}\wedge 1\right)\left(\frac{\mbox{dist}(y,\Omega^c)}{\sqrt t\wedge\mbox{diam}(\Omega^c)}\wedge 1\right)e^{-\frac{|x-y|^2}{ct}}.$$
However, for the heat kernel on exterior domains in the plane, 
Grigor'yan  and Saloff-Coste in \cite[Theorem 1.2]{gsc02}  
observed the heat kernel has an essentially different behavior. See the final paragraph of the introduction of \cite{kvz16}.  

Our proof above, after decomposing the function $g$ to $g\phi$ and $g(1-\phi)$, 
only needs to take care of the heat kernel $p_t^D(x,y)$ where $y$ stays away from the boundary, $y\in\rn\setminus B(0,2\delta R)$, 
where we have by Proposition \ref{heat-comparison} that 
for all $x,y\in \rr^n\setminus B(0,\delta R)$ and $t>0$,
it holds that
\begin{equation}\label{upper-difference-1}
0\le p_t(x,y)-p^D_t(x,y)\le  Ct^{-n/2}e^{-\frac{|x-y|^2+|x|^2+|y|^2}{ct}}.
\end{equation}
Note that the above inequality holds for $n=2$, too.
\end{rem}

\section{Riesz transform of the Dirichlet operator}
\hskip\parindent In this section, we study the behavior of the Riesz transform of the Dirichlet operators.

\subsection{A criteria for sublinear operators}
\hskip\parindent The following criteria is built upon recent developments for the study of Riesz transform
on manifolds with ends in \cite{ji20}. Note that since the Riesz transform on an exterior domain can not be bounded for $p\ge n$ when $n\ge 3$ and $p>2$ when $n=2$ by the examples provided by Hassell and Sikora  \cite[Proof of Theorem 5.6 \& Remark 5.8]{hs09} (see also \cite[Proposition 7.2]{kvz16} or Remark \ref{unbound-dirichlet} below), we only need to consider the case $2<p<n$. Recall that we always assume that the origin belongs to the interior of $\rn\setminus \Omega$.
\begin{thm}\label{criteria-Riesz}
Let $\Omega\subset\rn$ be an exterior Lipschitz domain, $n\ge 3$.
Suppose that $T$ is a sublinear operator that is bounded on $L^2(\Omega)$. Let $2<q<p<\infty$.
Assume that there exist $1<\alpha_1<\alpha_2<\infty$ such that  for all balls $B=B(x_B,r_B)$, $x_B\in\overline{\Omega}$, it holds that
\begin{equation}\label{condition2}
\left(\fint_{B}|T(f\chi_{\rn\setminus \alpha_2 B})|^p\,dx\right)^{1/p}\le
C\left\{\left(\fint_{\alpha_1B\cap\Omega}(|T(f\chi_{\rn\setminus \alpha_2 B})(x)|^2\,dx\right)^{1/2}+\frac{\|f\|_q}{(1+|x_B|+r_B)^{n/q}}\right\}.
\end{equation}
Then $T$ is weakly bounded on $L^q(\Omega)$.
\end{thm}
\begin{proof}
Let $f\in C^\infty_c(\Omega)$.  Then $Tf\in L^2(\Omega)$. We extend $Tf$ to $\rn$ by letting $Tf(x)=0$ outside of $\Omega$. By the $L^2$-boundedness of $T$, we have
\begin{equation}\label{condition1}
\left(\fint_{B}|T(f\chi_{\alpha_2B})|^2\,dx\right)^{1/2}\le C\left(\fint_{\alpha_2B}|f|^2\,dx\right)^{1/2}\le C\inf_{x\in B}\mathcal{M}_2(f)(x),
\end{equation}

For $\lambda>0$, let
$$E_\lambda:=\left\{x\in\rn:\, \mathcal{M}_2(|Tf|)(x)>\lambda\right\},$$
$$F_\lambda:=\left\{x\in\rn:\, \mathcal{M}_2(f)(x)>\lambda\right\},$$
and
$$G_{\lambda}:=\left\{x\in \rn:\, \frac{\|f\|_q}{(1+|x|)^{n/q}}>\lambda\right\}.$$

Note that $E_\lambda$ is an open set.
By the Vitali covering theorem, we can find a sequence of balls $\{B_j\}_j$,
which are of bounded overlap, such that $\{\frac 15B_j\}_j$ are disjoint,
\begin{equation}
 \bigcup_{j} \frac 15 B_j\subset E_\lambda\subset \cup_j B_j,
\end{equation}
and
$$B_j\cap (\rn\setminus E_\lambda)\neq\emptyset.$$

Let $K>1$, $\gamma>0$ to be fixed later. It holds obviously that
$$E_{K\lambda}\subset E_\lambda.$$
For each $j$, let
$$E_j:=\left\{x\in B_j:\, \mathcal{M}_2(|Tf|)(x)>K\lambda,\, \mathcal{M}_2(f)(x)\le \gamma\lambda,\,
\frac{\|f\|_q}{(1+|x|)^{n/q}}\le \lambda\right\}.$$
Note first that, for $x\in E_j$,
$$\mathcal{M}_2(|Tf|\chi_{3B_j})(x)>K\lambda.$$
In fact, by the definition of maximal functions, there exists a ball $B$ containing $x$ such that
$$\fint_{B}|Tf(y)|\,dy>K\lambda.$$
From this, we deduce that $B\subset E_\lambda$, and $r_B\le 2r_{B_j}$, since otherwise $r_B> 2r_{B_j}$ and $B_j\subset B$, which together with $B_j\cap (\rn\setminus E_\lambda)\neq\emptyset$ will imply that
$$\fint_{B}|Tf(y)|\,dy\le \lambda.$$
In particular, we have $B\subset 3B_j$, and
$$\mathcal{M}_2(|Tf|\chi_{3B_j})(x)>K\lambda,$$
which means
\begin{eqnarray*}
E_j&&=\left\{x\in B_j:\, \mathcal{M}_2(|Tf|\chi_{3B_j})(x)>K\lambda,\, \mathcal{M}_2(f)(x)\le \gamma\lambda,\,
\frac{\|f\|_q}{(1+|x|)^{n/q}}\le \lambda\right\}\\
&&\subset \left\{x\in B_j:\, \mathcal{M}_2(|T(f\chi_{\rn\setminus 3\alpha_2B_j})|\chi_{3B_j})(x)>\frac 12K\lambda,\, \mathcal{M}_2(f)(x)\le \gamma\lambda,\,
\frac{\|f\|_q}{(1+|x|)^{n/q}}\le \lambda\right\}\\
&&\quad\cup \left\{x\in B_j:\, \mathcal{M}_2(|T(f\chi_{3\alpha_2 B_j})|\chi_{3B_j})(x)>\frac 12K\lambda,\,\mathcal{M}_2(f)(x)\le \gamma\lambda\right\}\\
&&=:E_{j,1}\cup E_{j,2}.
\end{eqnarray*}
If $E_{j,2}\neq\emptyset$, then there exists $x\in E_{j,2}$ such that
$\mathcal{M}_2(f)(x)\le \gamma\lambda$. By \eqref{condition1}, we have
\begin{eqnarray*}
|E_{j,2}|\le \frac{C}{(K\lambda)^2}\int_{3B_j}|T(f\chi_{3\alpha_2B_j})|^2\chi_{3B_j}\,dy\le \frac{C|B_j|}{(K\lambda)^2}\inf_{x\in 3B_j}\mathcal{M}_2(f)(x)\le C\gamma^2K^{-2}|B_j|.
\end{eqnarray*}

By \eqref{condition2}, we have

\begin{eqnarray*}
|E_{j,1}|&&\le \frac{C}{(K\lambda)^p}\int_{3B_j}|T(f\chi_{\rn\setminus 3\alpha_2B_j})|^p\,dy\\
&&\le \frac{C|B_j|}{(K\lambda)^p}\left\{\left(\fint_{3\alpha_1B_j}|T(f\chi_{\rn\setminus 3\alpha_2B_j})|^2\,dy\right)^{p/2}+\frac{\|f\|^p_q}{(1+|x_{B_j}|+r_{B_j})^{pn/q}}\right\}\\
&&\le \frac{C|B_j|}{(K\lambda)^p}\left\{\left(\fint_{3\alpha_1B_j}|T(f)|^2\,dy\right)^{p/2}+
\left(\fint_{3\alpha_1 B_j}|T(f\chi_{3\alpha_2B_j})|^2\,dy\right)^{p/2}+\frac{\|f\|^p_q}{(1+|x_{B_j}|+r_{B_j})^{pn/q}}\right\}\\
&&\le \frac{C|B_j|}{(K\lambda)^p}\left\{\inf_{x\in B_j}\mathcal{M}_2(|T(f)|)(x)^{p/2}+
\left(\fint_{3\alpha_2B_j}|f(y)|^2\,dy\right)^{p/2}+\frac{\|f\|^p_q}{(1+|x_{B_j}|+r_{B_j})^{pn/q}}\right\}\\
&&\le \frac{C|B_j|}{(K\lambda)^p}\left\{\inf_{x\in B_j}\mathcal{M}_2(|T(f)|)(x)^{p/2}+
\inf_{x\in B_j}\mathcal{M}_2(f)(x)^{p/2}+\inf_{x\in E_j}\frac{\|f\|^p_q}{(1+|x|)^{pn/q}}\right\}\\
&&\le CK^{-p}|B_j|.
\end{eqnarray*}
From the estimates for $E_{j,1},\,E_{j,2}$ we deduce that
\begin{eqnarray*}
|E_{K\lambda}|&&\le \sum_j|E_j|+\left|\left\{x\in E_\lambda:\,\mathcal{M}_2(f)(x)>\gamma\lambda\right\}\right|+\left|\left\{x\in E_\lambda:\,\frac{\|f\|_q}{(1+|x|)^{n/q}}>\lambda\right\}\right|\\
&&\le \sum_j(|E_{j,1}|+|E_{j,2}|)+\left|\left\{x\in \rn :\,\mathcal{M}_2(f)(x)>\gamma\lambda\right\}\right|+\left|\left\{x\in \rn:\,\frac{\|f\|_q}{(1+|x|)^{n/q}}>\lambda\right\}\right|\\
&&\le C(\gamma^2K^{-2}+K^{-p})\sum_j|B_j|+\frac{C\|f\|_q^q}{(\gamma\lambda)^q}
+\frac{C\|f\|_q^q}{\lambda^q}\\
&&\le C(\gamma^2K^{-2}+K^{-p})|E_\lambda|+\frac{C\|f\|_q^q}{(\gamma\lambda)^q}
+\frac{C\|f\|_q^q}{\lambda^q}.
\end{eqnarray*}
Multiplying each side by $(K\lambda)^q$, we see that
\begin{eqnarray*}
(K\lambda)^q|E_{K\lambda}|&&\le  CK^q(\gamma^2K^{-2}+K^{-p})\lambda^q|E_\lambda|+CK^q(\gamma^{-q}+1)\|f\|_q^q
\end{eqnarray*}
Since $2<q<p$, by letting $K$ large enough first and  then $\gamma>0$ small enough such that
$$CK^q(\gamma^2K^{-2}+K^{-p})\le \frac 12,$$
we finally conclude that
\begin{eqnarray*}
\|Tf\|_{L^{q,\infty}(\Omega)}\le C\|f\|_{L^q(\Omega)},
\end{eqnarray*}
as desired.
\end{proof}

\subsection{Characterization for the Dirichlet case}
\hskip\parindent With the criteria Theorem \ref{criteria-Riesz} and the reverse inequality Theorem \ref{reverse-ineq} at hand, we can now finish the proof
for Theorem  \ref{main-dirichlet} following Shen \cite{shz05}.

\begin{prop}\label{main-dirichlet-1}
Let $\Omega\subset \rn$ be an exterior Lipschitz domain, $n\ge 3$. Let $p\in (2,n)$.
Suppose that there exist  $C>0$ and $1<\alpha_1<\alpha_2<\infty$  such that for any ball $B(x_0,r)$ satisfying
$B(x_0,\alpha_2r)\subset \Omega$ or $B(x_0,\alpha_2r)\cap\partial\Omega\neq \emptyset$, and any weak solution $u$ of $\mathscr{L}_Du=0$ in $\Omega\cap B(x_0,\alpha_2r)$, satisfying additionally $u=0$ on $B(x_0,\alpha_2r)\cap\partial\Omega$ if $x_0\in\partial \Omega$, it holds
$$\lf(\fint_{B(x_0,r)\cap\Omega}|\nabla u|^{p}\,dx\r)^{1/p}\le
\frac{C}{r}\fint_{B(x_0,\alpha_1r)\cap\Omega}|u|\,dx. \leqno({RH}_p)$$
Then the Riesz transform is bounded on $L^q(\Omega)$ for all $2<q<p$.
\end{prop}
\begin{proof}
For any $f\in C^\infty_c(\Omega,\rn)$, consider the solution
$u=\LV^{-1}_D \mathrm{div} f$.
Since the operator $T=\nabla \LV^{-1}_D \mathrm{div}$ is bounded on $L^2(\Omega)$,
we have $\nabla u\in L^2(\Omega)$.

For any ball $B=B(x_B,r_B)$ with $x_B\in\overline{\Omega}$, if $\alpha_2B\cap \rn\setminus \Omega=\emptyset$, then $\LV^{-1}_D \mathrm{div} (f\chi_{\Omega\setminus 4\alpha_2^2B})$ is harmonic in $\alpha_2B$.  We deduce from the Poincar\'e inequality that
\begin{eqnarray}\label{reg-har1}
\left(\fint_{B}|\nabla\LV_D^{-1} \mathrm{div} (f\chi_{\Omega\setminus 4\alpha_2^2B})|^p\,dx\right)^{1/p}&&\le
\inf_{c\in\rr}\frac{C}{r_B}\fint_{\alpha_1B}|\LV^{-1}_D \mathrm{div} (f\chi_{\Omega\setminus 4\alpha_2^2B})-c|\,dx\nonumber\\
&& \le C\left(\fint_{\alpha_1B}|\nabla\LV^{-1}_D \mathrm{div} (f\chi_{\Omega\setminus 4\alpha_2^2B})|^2\,dx\right)^{1/2}.
\end{eqnarray}

Now suppose that $B=B(x_B,r_B)$ and $\alpha_2B\cap \rn\setminus \Omega\neq\emptyset$. Since $\Omega$ has a compact Lipschitz boundary, there exists $r_0>0$ such that for any
$r<r_0$ and $x_0\in \partial\Omega$, the Poincar\'e inequality holds for $u\in W^{1,2}(B(x_0,r))$ that vanishes on $B(x_0,r)\cap \Omega^c$.

Suppose first $r_B<r_0/(\alpha_2+1)^2$, and choose $x_0\in \alpha_2B\cap \partial\Omega$. Then it holds
$$B\subset B(x_0,(\alpha_2+1)r_B)\subset B(x_0,\alpha_2(\alpha_2+1)r_B)\subset (\alpha_2+1)^2B\subset 4\alpha_2^2B.$$
The $(RH_p)$ condition together with the Poincar\'e inequality for small balls and functions
with vanishing boundary value implies that
\begin{eqnarray}\label{reg-har2}
\left(\fint_{B\cap\Omega}|\nabla\LV^{-1} \mathrm{div} (f\chi_{\rn\setminus 4\alpha_2^2B})|^p\,dx\right)^{1/p}&&\le
\left(\fint_{B(x_0,(\alpha_2+1)r_B)\cap \Omega}|\nabla\LV^{-1} \mathrm{div} (f\chi_{\rn\setminus 4\alpha_2^2B})|^p\,dx\right)^{1/p}\nonumber\\
&&\le \frac{C}{r_B}\fint_{B(x_0,\alpha_1(\alpha_2+1)r_B)\cap \Omega}|\LV^{-1} \mathrm{div} (f\chi_{\rn\setminus 4\alpha_2^2B})|\,dx\nonumber\\
&& \le C\left(\fint_{B(x_0,\alpha_1(\alpha_2+1)r_B)\cap\Omega}|\nabla\LV^{-1} \mathrm{div} (f\chi_{\rn\setminus 4\alpha_2^2B})|^2\,dx\right)^{1/2}\nonumber\\
&&\le C\left(\fint_{(\alpha_2+1)^2B\cap\Omega}|\nabla\LV^{-1} \mathrm{div} (f\chi_{\rn\setminus 4\alpha_2^2B})|^2\,dx\right)^{1/2}.
\end{eqnarray}

Suppose now $r_B\ge r_0/(\alpha_2+1)^2$. Since $\nabla \LV_D^{-1/2}$ is bounded on $L^r(\Omega)$ for all $1<r\le 2$,
the duality implies that $\LV_D^{-1/2}\mathrm{div}$ is bounded on $L^{q}(\Omega)$ for all $2\le q<\infty$.
The heat kernel bound of $e^{-t\LV_D}$ implies that the kernel of $\LV_D^{-1/2}$, given via the formula
\begin{equation}\label{squareroot}
\LV_D^{-1/2}=\frac{\sqrt{\pi}}{2}\int_0^\infty e^{-s\LV_D}\frac{\,ds}{\sqrt s},
\end{equation}
is bounded by
$$\frac{C}{|x-y|^{n-1}},$$
and hence $\LV_D^{-1/2}$ maps $L^q(\Omega)$ to $L^{q^\ast}(\Omega)$ for all $2< q<n$, where $q^\ast=\frac{nq}{n-q}$; see Stein \cite{ste70}. Moreover, since $\nabla\LV_D^{-1/2}$ is bounded on $L^s(\Omega)$ for $1<s\le 2$,
the dual operator $\LV_D^{-1/2}\text{div}$ is bounded on $L^{s'}(\Omega)$ for $1<s<2$.

We then deduce from the $(RH_p)$ condition, the H\"older inequality and
the mapping property of $\LV_D^{-1/2}$ and $\LV_D^{-1/2}\text{div}$ that for $2<q<n$,
\begin{eqnarray}\label{reg-har3}
\left(\fint_{B\cap \Omega}|\nabla\LV^{-1} \mathrm{div} (f\chi_{\Omega\setminus 4\alpha_2^2B})|^p\,dx\right)^{1/p}&&\le
 \frac{C}{r_B}\fint_{\alpha_1B}|\LV_D^{-1} \mathrm{div} (f\chi_{\Omega\setminus 4\alpha_2^2B})|\,dx\nonumber\\
&& \le \frac{C}{r_B}\left(\fint_{\alpha_1B}|\LV_D^{-1} \mathrm{div} (f\chi_{\Omega\setminus 4\alpha_2^2B})|^{q^\ast}\,dx\right)^{1/{q^\ast}}\nonumber\\
&&\le  \frac{C}{r_B|B|^{1/q^\ast}}\left(\int_{\Omega}| \LV_D^{-1/2}\mathrm{div} (f\chi_{\Omega\setminus 4\alpha_2^2B})|^q\,dx\right)^{1/q}\nonumber\\
&&\le \frac{C}{|B|^{1/q}}\left(\int_{\Omega}|f\chi_{\Omega\setminus 4\alpha_2^2B}|^q\,dx\right)^{1/q}\nonumber\\
&&\le \frac{C}{(1+|x_B|+r_B)^{n/q}}\|f\|_{L^q(\Omega)},
\end{eqnarray}
where the last inequality holds since $0\in \Omega^c$, $|x_B|\le r_B+\mathrm{diam}(\Omega^c)$ and $r_B\ge r_0/(\alpha_2+1)^2$.

Combining the estimates \eqref{reg-har1}, \eqref{reg-har2} and \eqref{reg-har3}, we see that
for $T=\nabla \LV_D^{-1} \mathrm{div}$, it holds for any $2<q<\min\{p,n\}$ that
\begin{equation}
\left(\fint_{B\cap\Omega}|T(f\chi_{\Omega\setminus 4\alpha_2^2B})|^p\,dx\right)^{1/p}\le
C\left\{\left(\fint_{(\alpha_2+1)^2B\cap\Omega}(|T(f\chi_{\Omega\setminus 4\alpha_2^2B})(x)|^2\,dx\right)^{1/2}+\frac{\|f\|_q}{(1+|x_B|+r_B)^{n/q}}\right\}.
\end{equation}
Consequently, Theorem \ref{criteria-Riesz} implies that $T=\nabla \LV^{-1} \mathrm{div}$
is weakly bounded on $L^q(\Omega)$ for all $2<q<p$. The Marcinkiewicz interpolation theorem then implies that
$T=\nabla \LV^{-1}_D \mathrm{div}$  is bounded on $L^q(\Omega)$ for all $2<q<p$.

It is obvious that $T=\nabla \LV_D^{-1} \mathrm{div}$ is a self-adjoint operator, which implies via a duality argument that $T=\nabla \LV^{-1}_D \mathrm{div}$ is bounded on $L^{q'}(\Omega)$, where $q'$ is the H\"older conjugate of $q$, $2<q<n$.

Theorem \ref{reverse-ineq} then implies that
\begin{eqnarray*}
\|\LV^{-1/2}_D \mathrm{div}\|_{L^{q'}(\Omega)\to L^{q'}(\Omega)}=\|\LV_D^{1/2}\LV_D^{-1} \mathrm{div}\|_{L^{q'}(\Omega)}\le
\|\nabla\LV^{-1}_D \mathrm{div}\|_{L^{q'}(\Omega)}\le C.
\end{eqnarray*}
Once again, a duality argument implies that
\begin{eqnarray*}
\|\nabla \LV^{-1/2}_D \|_{L^{q}(\Omega)\to L^{q}(\Omega)}\le C,
\end{eqnarray*}
for all $2<q<p$, as desired.
\end{proof}

We have the following open-ended property for $(RH_p)$ (see \cite[Lemma 3.11]{ji20}).
\begin{prop}\label{open-ended}
Let $\Omega\subset \rn$ be an exterior Lipschitz domain, $n\ge 3$. Let $p\in (2,n)$.
Then the condition $(RH_p)$ is open ended, i.e., there exists $\epsilon>0$ such that
$p+\epsilon<n$ such that $(RH_{p+\epsilon})$ holds.
\end{prop}
\begin{proof}
If $B(x_0,\alpha_2r)\subset \Omega$, then similar to \eqref{reg-har1},
the Poincar\'e inequality implies that
$$\lf(\fint_{B(x_0,r)}|\nabla u|^{p}\,dx\r)^{1/p}\le
C\lf(\fint_{B(x_0,\alpha_1r)}|\nabla u|^{2}\,dx\r)^{1/2}. $$
The Gehring Lemma (cf. \cite{ge73,bcf14}) then implies there exists $\epsilon_1>0$ such that
\begin{eqnarray*}
\lf(\fint_{B(x_0,r)}|\nabla u|^{p+\epsilon_1}\,dx\r)^{1/(p+\epsilon_1)}&&\le
C\lf(\fint_{B(x_0,\alpha_1r)}|\nabla u|^{2}\,dx\r)^{1/2}.
\end{eqnarray*}
Choose a sequence of balls $\{B(x_j,(\alpha_1-1)r/(4\alpha_1))\}_{1\le j\le c(n)}$ such that $\{B(x_j,(\alpha_1-1){r}/({8\alpha_1}))\}_j$
are disjoint, $x_j\in B(x_0,r)$ and $B(x_0,r)\subset \cup_j B(x_j,(\alpha_1-1)r/(4\alpha_1))$. For each $j$,
the Caccioppoli inequality implies that
\begin{eqnarray}\label{covering-1}
\lf(\fint_{B(x_j,(\alpha_1-1)r/(4\alpha_1)))}|\nabla u|^{p+\epsilon_1}\,dx\r)^{1/(p+\epsilon_1)}&&\le
C\lf(\fint_{B(x_j,(\alpha_1-1)r/4)}|\nabla u|^{2}\,dx\r)^{1/2}\nonumber\\
&&\le \frac{C}{r}\lf(\fint_{B(x_j,(\alpha_1-1)r/2)}|u|^{2}\,dx\r)^{1/2}\nonumber\\
&&\le \frac{C}{r}\|u\|_{L^\infty(B(x_j,(\alpha_1-1)r/2))}\nonumber\\
&&\le \frac{C}{r}\lf(\fint_{B(x_j,(\alpha_1-1)r)}|u|\,dx\r)\nonumber\\
&&\le  \frac{C}{r}\lf(\fint_{B(x_j,(\alpha_1-1)r)}|u|\,dx\r).
\end{eqnarray}
A summation over $j$ gives that
\begin{eqnarray}\label{covering-2}
\quad\quad\lf(\fint_{B(x_0,r)}|\nabla u|^{p+\epsilon_1}\,dx\r)^{1/(p+\epsilon_1)}&&\le C\sum_{j=1}^{c(n)}\lf(\fint_{B(x_j,(\alpha_1-1)r)}|\nabla u|^{p+\epsilon_1}\,dx\r)^{1/(p+\epsilon_1)}\nonumber\\
&&\le \sum_{j=1}^{c(n)}\frac{C}{r}\lf(\fint_{B(x_j,(\alpha_1-1)r)}|u|\,dx\r)\nonumber\\
&&\le  \frac{C}{r}\lf(\fint_{B(x_0,\alpha_1r)}|u|\,dx\r).
\end{eqnarray}

For balls $B(x_0,\alpha_2r)\cap\partial\Omega\neq \emptyset$ with $r<r_0/(\alpha_2+1)^2$, the same argument as in
\eqref{reg-har2} yields that for weak solution $u$ of $\mathscr{L}_Du=0$ in $\Omega\cap B(x_0,\alpha_2r)$, satisfying additionally $u=0$ on $B(x_0,\alpha_2r)\cap\partial\Omega$,
it holds
$$\lf(\fint_{B(x_0,r/(\alpha_2+1)^2)\cap\Omega}|\nabla u|^{p}\,dx\r)^{1/p}\le
C\lf(\fint_{\alpha_1B(x_0,r)\cap\Omega}|\nabla u|^{2}\,dx\r)^{1/2}. $$
Apparently, the last inequality holds also for all
balls that $B(x_0,\alpha_2r)\cap\Omega=\emptyset$.

Since $\Omega$ is a doubling space, by applying the Gehring Lemma (cf. \cite{ge73,bcf14}) again, we deduce that there exists $\epsilon_2>0$ such that $$\lf(\fint_{B(x_0,r/(\alpha_2+1)^2)\cap\Omega}|\nabla u|^{p+\epsilon_2}\,dx\r)^{1/(p+\epsilon_2)}\le
C\lf(\fint_{\alpha_1B(x_0,r)\cap\Omega}|\nabla u|^{2}\,dx\r)^{1/2},$$
for balls with $r<r_0/(\alpha_2+1)^2$ and harmonic functions $u$ on $B(x_0,\alpha_2r)$ with $u=0$ on the boundary
$\partial\Omega\cap B(x_0,12r)$ if the set is not empty. By  repeating the argument as in \eqref{covering-1} and
\eqref{covering-2}, and using Caccioppoli inequality, we can conclude that
\begin{eqnarray}\label{covering-3}
\lf(\fint_{B(x_0,r)\cap\Omega}|\nabla u|^{p+\epsilon_2}\,dx\r)^{1/(p+\epsilon_2)}
\le  \frac{C}{r}\lf(\fint_{\alpha_1B(x_0,r)\cap \Omega}|u|\,dx\r).
\end{eqnarray}

For balls $B(x_0,r)$ with $r_0/(\alpha_2+1)^2\le r\le R_0$, where $R_0=3\mathrm{diam}(\rn\setminus \Omega)$ and $x_0\in\partial \Omega$, by dividing $B(x_0,r)\cap \Omega$ into the union of small balls $\{B(x_j,r_0/[2(\alpha_2+1)^2])\cap\Omega\}_{1\le j\le c(n,r_0,R_0)}$ and repeating the argument as in \eqref{covering-1} and
\eqref{covering-2} yields the desired estimate as
\begin{eqnarray}\label{covering-4}
\lf(\fint_{B(x_0,r)\cap\Omega}|\nabla u|^{p+\epsilon_2}\,dx\r)^{1/(p+\epsilon_2)}
\le  \frac{C}{r}\lf(\fint_{\alpha_1B(x_0,r)\cap \Omega}|u|\,dx\r).
\end{eqnarray}

Finally, let us consider balls $B(x_0,r)$ with $r>R_0$ and $x_0\in\partial\Omega$. Let $0<\epsilon\le \min\{\epsilon_1,\epsilon_2\}$ be such that $p+\epsilon<n$.
For $x\in B(x_0,r)\setminus B(x_0,R_0)$, it holds
$$\mathrm{diam}(\rn\setminus \Omega)=R_0/3<R_0/2\le |x-x_0|/2<r/2,$$
and therefore, the ball $B(x,|x-x_0|/(2\alpha_2))$ satisfies that
$\alpha_2B(x,|x-x_0|/(2\alpha_2))=B(x,|x-x_0|/2)$ does not intersect $\rn\setminus \Omega$
and $B(x,|x-x_0|/2)\subset B(x_0,4r/3)$. \eqref{covering-2} implies that
\begin{eqnarray*}
\lf(\fint_{B(x,\frac{(\alpha_1-1)|x-x_0|}{2\alpha_1\alpha_2})}|\nabla u|^{p+\epsilon}\,dy\r)^{1/(p+\epsilon)}&&\le \frac{C}{|x-x_0|}\lf(\fint_{B(x,\frac{(\alpha_1-1)|x-x_0|}{2\alpha_2})}|u|\,dy\r)\\
&&\le \frac{C}{|x-x_0|}\|u\|_{L^\infty(B(x_0,r+\frac{(\alpha_1-1)r}{2\alpha_2})\cap\Omega)}\\
&&\le \frac{C}{|x-x_0|}\lf(\fint_{B(x_0,\alpha_1r)\cap\Omega}|u|\,dy\r),
\end{eqnarray*}
and hence,
\begin{eqnarray*}
\fint_{B(x,\frac{(\alpha_1-1)|x-x_0|}{2\alpha_1\alpha_2})}|\nabla u|^{p+\epsilon}\,dy&&\le
\frac{C}{|x-x_0|^{p+\epsilon}}\lf(\fint_{B(x_0,\alpha_1r)\cap\Omega}|u|\,dy\r)^{p+\epsilon}.
\end{eqnarray*}
Integrating over $B(x_0,r)\setminus B(x_0,R_0)$ yields that
\begin{eqnarray}\label{covering-5}
\int_{B(x_0,r)\setminus B(x_0,R_0)}\fint_{B(x,\frac{(\alpha_1-1)|x-x_0|}{2\alpha_1\alpha_2})}|\nabla u|^{p+\epsilon}\,dy\,dx
&&\le \int_{B(x_0,r)\setminus B(x_0,R_0)}\frac{C}{|x-x_0|^{p+\epsilon}}\lf(\fint_{B(x_0,\alpha_1r)\cap\Omega}|u|\,dy\r)^{p+\epsilon}\,dx\nonumber\\
&&\le Cr^{n-p-\epsilon}\lf(\fint_{B(x_0,\alpha_1r)\cap\Omega}|u|\,dy\r)^{p+\epsilon},
\end{eqnarray}
since $p+\epsilon<n$. Note that for $y\in B(x,\frac{(\alpha_1-1)|x-x_0|}{2\alpha_1\alpha_2})$, it holds
$$\frac{2\alpha_1\alpha_2-(\alpha_1-1)}{2\alpha_1\alpha_2}|x-x_0|\le |y-x_0|\le \frac{2\alpha_1\alpha_2+(\alpha_1-1)}{2\alpha_1\alpha_2}|x-x_0|.$$
By the Fubini theorem, the LHS of \eqref{covering-5} satisfies
\begin{eqnarray*}
&&\int_{B(x_0,r)\setminus B(x_0,R_0)}\fint_{B(x,\frac{(\alpha_1-1)|x-x_0|}{2\alpha_1\alpha_2})}|\nabla u|^{p+\epsilon}\,dy\,dx\\
&&\quad\ge C\int_{B(x_0,r)\setminus B(x_0,R_0)}\int_{B(x_0,r)\setminus B(x_0,R_0)}\frac{|\nabla u(y)|^{p+\epsilon}}{|x-x_0|^n}
\chi_{B(x,\frac{(\alpha_1-1)|x-x_0|}{2\alpha_1\alpha_2})}(y)\,dy\,dx\\
&&\quad \ge C\int_{B(x_0,r)\setminus B(x_0,R_0)}\int_{B(x_0,r)\setminus B(x_0,R_0)}\frac{|\nabla u(y)|^{p+\epsilon}}{|y-x_0|^n}
\chi_{B(y,\frac{(\alpha_1-1)}{2\alpha_1\alpha_2+(\alpha_1-1)}|y-x_0|)}(x)\,dx\,dy\\
&&\quad \ge C\int_{B(x_0,r)\setminus B(x_0,R_0)}|\nabla u(y)|^{p+\epsilon}\,dy,
\end{eqnarray*}
which together with \eqref{covering-5} implies that
\begin{eqnarray*}
\int_{B(x_0,r)\setminus B(x_0,R_0)}|\nabla u|^{p+\epsilon}\,dx
&&\le Cr^{n-p-\epsilon}\lf(\fint_{B(x_0,\alpha_1r)\cap\Omega}|u|\,dy\r)^{p+\epsilon}.
\end{eqnarray*}
On the other hand, \eqref{covering-4} together with $\epsilon\le \epsilon_2$, $p+\epsilon<n$ yields that
\begin{eqnarray*}
\int_{B(x_0,R_0)\cap\Omega}|\nabla u|^{p+\epsilon}\,dx
&&\le  CR_0^{n-p-\epsilon}\lf(\fint_{\alpha_1B(x_0,R_0)\cap \Omega}|u|\,dx\r)^{p+\epsilon}\\
&&\le Cr^{n-p-\epsilon}\|u\|_{L^\infty(\alpha_1B(x_0,R_0)\cap \Omega)}\\
&&\le Cr^{n-p-\epsilon}\lf(\fint_{\alpha_1B(x_0,r)\cap\Omega}|u|\,dy\r)^{p+\epsilon}.
\end{eqnarray*}
The above two estimates imply that
\begin{eqnarray}\label{covering-6}
\left(\fint_{B(x_0,r)\cap \Omega}|\nabla u|^{p+\epsilon}\,dx\right)^{1/(p+\epsilon)}
&&\le \frac{C}{r}\lf(\fint_{\alpha_1B(x_0,r)\cap\Omega}|u|\,dy\r),
\end{eqnarray}
which together with \eqref{covering-2}, \eqref{covering-3}, \eqref{covering-4} completes the proof.
\end{proof}

As a byproduct of the proof of the previous proposition, we have the following
 characterization for the condition $(RH_p)$ for $p<n$.
\begin{prop}\label{character-rhp}
Let $\Omega\subset \rn$ be an exterior Lipschitz domain, $n\ge 3$. Let $p\in (2,n)$.
Then for $2<p<n$, the condition $(RH_p)$ is equivalent to that,
$(RH_p)$ holds on  balls  $B(x_0,r)$ satisfying either $B(x_0,\alpha_2r)\subset \Omega$, or  $x_0\in\partial \Omega$ and $r<r_0$ for some $0<r_0<\infty$, for some  $\alpha_2>1$.
\end{prop}
\begin{proof}
One side is obvious, for the other side, note that we only need to verify that for balls
$B(x_0,3r)\cap\partial\Omega\neq \emptyset$ with $x_0\in\partial \Omega$ and  $r\ge r_0$ the inequality
$(RH_p)$ holds. From the estimates \eqref{covering-4} and \eqref{covering-6}, we see that this required estimates follows from the corresponding estimates on balls $B(x_0,r)$ satisfying either $B(x_0,3r)\subset \Omega$, or $B(x_0,3r)\cap\partial\Omega\neq \emptyset$ with $x_0\in\partial \Omega$ and $r<r_0$ for some $0<r_0<\infty$.
\end{proof}

We can now finish the proof of Theorem \ref{main-dirichlet}.
\begin{proof}[Proof of Theorem \ref{main-dirichlet}]
For the implication $(ii)\Longrightarrow (i)$, Proposition \ref{open-ended} shows that
$(RH_p)$ implies $(RH_{p+\epsilon})$ for some $\epsilon>0$ such that $p+\epsilon<n$, which together with
Proposition \ref{main-dirichlet-1} yields that the Riesz transform is bounded on $L^q(\Omega)$
for all $2<q<p+\epsilon$, in particular, bounded on $L^p(\Omega)$.

Let us prove the converse side $(i)\Longrightarrow (ii)$. Suppose that $\LV_D u=0$ in $3B\cap \Omega$, $x_B\in \Omega$. We further assume that
$u=0$ on $3B\cap \partial\Omega$, if the set is not empty.

Choose  $\varphi\in C^\infty_c(2B)$ with $\varphi=1$ on $B$, and $|\nabla\varphi|\le C/r_B$.
It holds in the distribution sense, that
$$\LV_D (u\varphi)=-\mathrm{div}(uA\nabla\varphi)-A\nabla u\cdot \nabla \varphi.$$
The boundedness of  Riesz transform $\nabla \LV_D^{-1/2}$ on $L^q(\Omega)$ for $1<q<2$ and on $L^p(\Omega)$, where $2<p<n$, implies that   the operator
$\nabla\LV^{-1}_D\mathrm{div}$ is bounded on $L^p(\Omega)$, and by \eqref{squareroot} $\LV^{-1/2}$
maps $L^q(\Omega)$ to $L^{q^\ast}(\Omega)$ for $1<q<n$ and $q^\ast=\frac{nq}{n-q}$.
We therefore deduce that
\begin{eqnarray}\label{est-poisson-1}
\left(\fint_{B\cap\Omega}|\nabla u|^p\,dx\right)^{1/p}&&\le \left(\fint_{B\cap\Omega}|\nabla \LV^{-1}_D\mathrm{div}(uA\nabla\varphi)|^p\,dx\right)^{1/p}+\left(\fint_{B\cap\Omega}|\nabla \LV_D^{-1}(A\nabla u\cdot \nabla\varphi)|^p\,dx\right)^{1/p}\nonumber\\
&&\le C\left(\frac{1}{|B|}\int_{\Omega}|uA\nabla\varphi|^p\,dx\right)^{1/p}
+C\left(\frac{1}{|B|}\int_{\Omega}|\LV^{-1/2}(A\nabla u\cdot \nabla\varphi)|^p\,dx\right)^{1/p}\nonumber\\
&&\le \frac{C}{r_B} \left(\fint_{2B\cap\Omega}|u|^p\,dx\right)^{1/p}+\frac{C}{|B|^{1/p}}\left(\int_{\Omega}|\nabla u\cdot \nabla\varphi|^{p_\ast}\,dx\right)^{1/p_\ast}\nonumber\\
&&\le \frac{C}{r_B}\|u\|_{L^\infty(2B\cap\Omega)}+C\left(\fint_{2B\cap\Omega}|\nabla u|^{p_\ast}\right)^{1/p_\ast},
\end{eqnarray}
where $p_\ast=\frac{np}{n+p}$. After several steps of iteration, if necessary, we can conclude
via the Caccioppoli inequality  that
 \begin{eqnarray*}
\left(\fint_{B\cap\Omega}|\nabla u|^p\,dx\right)^{1/p}&&\le \frac{C}{r_B}\|u\|_{L^\infty(2^kB\cap\Omega)}+C\left(\frac{1}{|B|}\int_{2^kB}|\nabla u|^{2}\,dx\right)^{1/2}\\
&&\le \frac{C}{r_B}\|u\|_{L^\infty(2^kB\cap\Omega)}+\frac{C}{r_B}\fint_{2^{k+1}B\cap\Omega}|u|\,dx\\
&&\le \frac{C}{r_B}\fint_{2^{k+1}B\cap\Omega}|u|\,dx.
\end{eqnarray*}
A simple covering argument as in \eqref{covering-1} and \eqref{covering-2} implies that
 \begin{eqnarray*}
\left(\fint_{B\cap\Omega}|\nabla u|^p\right)^{1/p}&&\le  \frac{C}{r_B}\fint_{2B\cap\Omega}|u|\,dx,
\end{eqnarray*}
which completes the proof.
\end{proof}

\subsection{Dirichlet operators with VMO coefficients}
\hskip\parindent In this part, we give the proof of Theorem \ref{app-dirichlet}.
 Recall that
$$p_\LV:=\sup\{p>2: \,\nabla \LV^{-1/2}\ \text{is bounded on}\, L^p(\rn)\}.$$
\begin{proof}[Proof of Theorem \ref{app-dirichlet}]
Let us first prove the Lipschitz case (i). According to \cite[Theorem 1.1]{shz05} (see also \cite[Theorem 1.9]{cjks16}),
for any $2<p<p_\LV$, it holds for all harmonic functions $u$, $\LV u=0$ in $B(x_0,3r)$, that
$$\left(\fint_{B(x_0,r)}|\nabla u|^p\,dx\right)^{1/p}\le \frac{C}{r}\fint_{B(x_0,2r)}|u|\,dx.$$
This implies that, if $B(x_0,3r)\subset \Omega$, then for all harmonic functions $u$, $\LV_D u=0$ in $B(x_0,3r)$,
$$\left(\fint_{B(x_0,r)}|\nabla u|^p\,dx\right)^{1/p}\le \frac{C}{r}\fint_{B(x_0,2r)}|u|\,dx.$$

Suppose now $B(x_0,3r)\cap\partial\Omega\neq \emptyset$ with $x_0\in\partial \Omega$.
By \cite[Theorem B \& Theorem C]{shz05}, there exist $\epsilon>0$ and $r_0>0$ such that
for any $2<p<3+\epsilon$, if $r<r_0$, then for any
$u$ satisfying $\LV_D u=0$ and $u=0$ on $B(x_0,3r)\cap\partial\Omega$, it holds
$$\left(\fint_{B(x_0,r)}|\nabla u|^p\,dx\right)^{1/p}\le \frac{C}{r}\fint_{B(x_0,2r)}|u|\,dx.$$
By Proposition \ref{character-rhp}, this implies that
for $p<\min\{n,p_\LV,3+\epsilon\}$, $(RH_p)$ holds. Thus (i) follows
from Theorem \ref{main-dirichlet}.

For the $C^1$ domain case, we simply note that the above $\epsilon$ can be taken as $\infty$;
see \cite[Remark 4.5]{shz05}. Proposition \ref{character-rhp} then implies that $(RH_p)$ holds
for $p<\min\{n,p_\LV\}$. Theorem \ref{main-dirichlet} gives the desired result.
\end{proof}

For the completeness we wish to include an example for the case $p\ge n$. To this end, we recall the following result,
which should be known somewhere, but we were not able to find an exact reference.
\begin{lem}\label{sobolev}
Let $\Omega=\rn\setminus\overline{B(0,1)}$. Suppose that $u(x)=1-|x|^{2-n}$ when $n\ge 3$, $u(x)=\log|x|$ when $n=2$.
Then

(i) when $n=2$, $u\in \dot{W}^{1,p}_0(\Omega)$ for $p>n$, and $u\notin \dot{W}^{1,p}_0(\Omega)$ if $1\le p\le n$;

(ii) when $n\ge 3$,  $u\in \dot{W}^{1,p}_0(\Omega)$ for $p\ge n$, and $u\notin \dot{W}^{1,p}_0(\Omega)$ if $1\le p< n$.
\end{lem}
\begin{proof}
Let us first show that $u\in \dot{W}^{1,p}_0(\Omega)$ for $p>n$ for all $n\ge 2$.
For large enough $R$, by choosing a cut-off function $\psi_R\in C^\infty_c(\rn)$ such that $\supp\psi_R\subset B(0,2R)$,
$\psi_R=1$ on $B(0,R)$ and $|\nabla \psi_R|\le C/R$, it holds for $p>n$ that
\begin{eqnarray*}
\int_{\Omega}|\nabla (u-u\psi_R)|^p\,dx&&\le
\begin{cases}
\int_{\rr^n\setminus B(0,R)}\frac{C}{|x|^p}\,dx+\int_{B(0,2R)\setminus B(0,R)} \frac{C|\log |x||^p}{R^p}\,dx, & \ n=2\\
\int_{\rr^n\setminus B(0,R)}\frac{C}{|x|^{(n-1)p}}\,dx+\int_{B(0,2R)\setminus B(0,R)} \frac{C}{R^p}\,dx, & \ n\ge 3 \\
\end{cases}\\
&&\le \begin{cases}
CR^{2-p}+CR^{2-p}\log R, & \ n=2\\
CR^{n-(n-1)p}+CR^{n-p}, & \ n\ge 3, \\
\end{cases}
\end{eqnarray*}
which tends to zero as $R\to \infty$. Thus $u\in \dot{W}^{1,p}_0(\Omega)$ for $p>n$.

Suppose now $n\ge 3$. For large enough $R$, we let
$$\psi_R(x)=\begin{cases} 1 & \ 1\le |x|<R\\
\frac{\log\frac{R^2}{|x|}}{\log R}, & \ R\le |x|\le R^2\\
0,&\ |x|>R^2.
\end{cases}
$$
Then  $\psi_R$ is a compactly supported Lipschitz function. It holds that
\begin{eqnarray*}
\int_{\Omega}|\nabla (u-u\psi_R)|^n\,dx&&\le
\int_{\rr^n\setminus B(0,R)}\frac{C}{|x|^{(n-1)n}}\,dx+\int_{B(0,R^2)\setminus B(0,R)} \frac{C}{|x|^n\log^n R}\,dx\\
&&\le CR^{-(n-1)^2}+C(\log R)^{1-n}
\end{eqnarray*}
which tends to zero as $R\to \infty$. Thus $u\in \dot{W}^{1,n}_0(\Omega)$ for $n\ge 3$.

For the case $1\le p\le n$, $n=2$ and  $1\le p\le \frac{n}{n-1}$, $n\ge 3$, we have
\begin{eqnarray*}
\int_{\Omega}|\nabla u|^p\,dx=\int_{\rr^n\setminus B(0,1)} C|x|^{-p(n-1)}\,dx =\infty,
\end{eqnarray*}
which implies that $u\notin \dot{W}^{1,p}_0(\Omega)$.

For the case $\frac{n}{n-1}<p<n$, $n\ge 3$. Note that for any $v\in C^\infty_c(\Omega)\subset C^\infty_c(\rn)$, the Sobolev inequality gives
\begin{eqnarray*}
\int_{\rn}|v|^{\frac{np}{n-p}}\,dx\le C(n,p)\left(\int_{\rn} |\nabla v|^p\,dx\right)^{\frac{n}{n-p}}.
\end{eqnarray*}
Suppose there exists $v_k\in C^\infty_c(\Omega)$ such that
$$\|\nabla (u-v_k)\|_{L^p(\Omega)}\to 0,\ k\to \infty.$$
Extend $u,v_k$ to $\Omega^c=\overline{B(0,1)}$ as zero, then the Sobolev inequality implies that
\begin{eqnarray*}
\int_{\rn}|u|^{\frac{np}{n-p}}\,dx\le C(n,p)\left(\int_{\rn} |\nabla u|^p\,dx\right)^{\frac{n}{n-p}}
\end{eqnarray*}
holds also true for $u$. This contradicts with that $u$ itself is not $L^q$-integrable for any $1\le q<\infty$.
\end{proof}
\begin{remark}\label{unbound-dirichlet}\rm
Let $\Omega=\rn\setminus\overline{B(0,1)}$. Recall that Hassell and Sikora \cite{hs09} already discovered that $\nabla\Delta_D^{-1/2}$ on $\Omega$ is {\em not} bounded on $L^p$ for $p>2$ if $n=2$, and $p\ge n$ if $n\ge 3$; see \cite[Theorem1.1 \& Remark 5.8]{hs09}, and also \cite[Proposition 7.2]{kvz16} for $n\ge 3$.

Let $u(x)=1-|x|^{2-n}$ when $n\ge 3$, and $u(x)=\log|x|$ when $n=2$. Lemma \ref{sobolev} tells us that, $u\in \dot{W}^{1,p}_0(\Omega)$ for $p>n=2$ or $p\ge n\ge 3$.

Let us show that $u$ belongs to the null space of $\Delta_D^{1/2}$.
By \cite[Lemma 2.3]{cd99} there exists $\gamma>0$ such that for all $t>0$,
 \begin{eqnarray*}
 &&\int_{\Omega}|\nabla_x p^D_t(x,y)|^{2}\exp\left\{\gz |x-y|^2/t\right\}\,dx\le Ct^{1-\frac n2},
 \end{eqnarray*}
 which implies for $1<q<2$ that
  \begin{eqnarray*}
 &&\int_{\Omega}|\nabla_x p^D_t(x,y)|^{q}\exp\left\{\gz |x-y|^2/(2t)\right\}\,dx\le Ct^{\frac q2-\frac n2}.
 \end{eqnarray*}
 Thus $p_t^D(x,\cdot)\in {W}_0^{1,q}(\Omega)$ for $1\le q\le 2$, for all $t>0$. 
Therefore, for each $t>0$, $\Delta_De^{-t\Delta_D}u$ satisfies
\begin{eqnarray*}
\Delta_De^{-t\Delta_D}u&&=\int_\Omega(\Delta_D)_xp_t^D(x,y)u(y)\,dy=\int_\Omega(\Delta_D)_yp_t^D(x,y)u(y)\,dy\\
&&=-\int_\Omega\nabla_yp_t^D(x,y)\nabla u(y)\,dy=\int_\Omega p_t^D(x,y)\Delta u(y)\,dy=0,
\end{eqnarray*}
where the second equality  by symmetry of the heat kernel, the third equality by $u\in \dot{W}_0^{1,p}(\Omega)$
and $p_t^D(x,\cdot)\in {W}_0^{1,p'}(\Omega)$, $1/p+1/p'=1$, $p>n$ when $n=2$ and $p\ge n$ when $n\ge 3$. We thus see that
$$
\Delta_D^{1/2}u=\frac{1}{\sqrt{\pi}} \int_{0}^{\infty} \Delta_D e^{-s \Delta_D} u\frac{\,ds}{\sqrt{s}}
=\frac{1}{\sqrt{\pi}} \int_{0}^{\infty} e^{-s \Delta_D} \Delta_D u\frac{\,ds}{\sqrt{s}}=0.
$$

Note if the Riesz transform $\nabla \Delta_D^{-1/2}$ is $L^p$ bounded, then it implies that
$$\|\nabla v\|_{L^p(\Omega)}\le C\|\Delta_D^{1/2}v\|_{L^p(\Omega)},\,\forall \ v\in \dot{W}^{1,p}_0(\Omega),$$
see \eqref{app-riesz} and also \cite{ac05,acdh,at98,at01,kvz16}.
This implies that the Riesz transform $\nabla \Delta_D^{-1/2}$ is {\em not} $L^p$ bounded for $p>n=2$ and $p\ge n$, $n\ge 3$,
as otherwise it will hold that
$$0<\|\nabla u\|_{L^p(\Omega)}\le C\|\Delta_D^{1/2}u\|_{L^p(\Omega)}=0.$$
Note also the above example does not apply to the case $p\le n=2$ or $p< n$ when $n\ge 3$, since
$u\notin \dot{W}^{1,p}_0(\Omega)$ by Lemma \ref{sobolev}.
\end{remark}

\section{Riesz transform of the Neumann operator}
\hskip\parindent In this part, let us study the case of Neumann boundary conditions.

\subsection{Characterization for the Neumann case}
\hskip\parindent  For the heat kernel $p^N_t(x,y)$ of the semigroup $e^{-t\LV_N}$, Gyrya and
Saloff-Coste \cite[Chapter 3]{gsa11} shows that $p^N_t(x,y)$ satisfies the two side Gaussian bounds,
if $\Omega$ satisfies an inner uniform condition, which we recall below.

Consider the intrinsic distance given by
$$\rho_\Omega(x,y):=\sup\left\{f(x)-f(y):\, f\in W^{1,2}(\Omega)\cap C_c(\Omega),\, |\nabla f|\le 1\mbox{ a.e.}\right\},$$
where $C_c(\Omega)$ denotes the space of compactly supported continuous functions in $\Omega$.
The space $({\Omega},\rho_\Omega)$ satisfies the inner uniform condition, if
there exist $C,c>0$ such that  for any $x,y\in {\Omega}$, there is a rectifiable curve
$$\gamma:\,[0,1]\to \Omega$$
of length at most $C\rho_\Omega(x,y)$, connecting $x$ to $y$, satisfying
that
\begin{equation}\label{inner-uniform}
\dist(z,\partial\Omega)\ge c\frac{\rho_\Omega(x,z)\rho_\Omega(y,z)}{\rho_\Omega(x,y)},\ \, \forall\,z\in\gamma([0,1]).
\end{equation}
The definition generalized the uniform condition on Euclidean spaces introduced by
Martio and Sarvas  \cite{ms79} (see also  \cite{hk91,jo81}). Recall that
a domain is uniform, if there exist $C,c>0$ such that for any $x,y\in {\Omega}$, there is a rectifiable curve
$$\gamma:\,[0,1]\to \Omega$$
of length at most $C|x-y|$, connecting $x$ to $y$, satisfying
that
\begin{equation}\label{uniform}
\dist(z,\partial\Omega)\ge c\frac{|x-z||y-z|}{|x-y|},\ \, \forall\,z\in\gamma([0,1]).
\end{equation}
Further, $\Omega$ is a locally uniform domain, if there exists $r_0>0$ such that \eqref{uniform}
holds for any two points $x,y\in\Omega$ with $|x-y|<r_0$.

\begin{prop}\label{inn-uniform}
Suppose that $\Omega$ is an exterior Lipschitz domain. Then $(\Omega,\rho_\Omega)$ is inner uniform.
\end{prop}
\begin{proof}
Let us first show that an exterior Lipschitz domain is a uniform domain, which is a consequence of \cite[Theorem 3.4]{hk91}.

Since $\Omega$ is a Lipschitz domain and $\partial\Omega$ is a compact set,
there exists a small enough $r_0>0$ depending on the shape of $\partial\Omega$, such that for any $x_0\in\Omega\cap B(0,R)$, it holds  either $B(x_0,r_0)\subset\Omega$
or $B(x_0,r_0)\cap\Omega$ lie above (up to a rotation) a Lipschitz graph of part of the boundary.
Thus, $\Omega\cap B(0,R)$ is an $(\epsilon,\delta)$-uniform domain in the sense of Jones \cite{jo81},
and is locally uniform in the sense of Herron-Koskela \cite{hk91}.
Since $\partial\Omega$ is compact, \cite[Theorem 3.4]{hk91} shows
that locally uniform is equivalent to uniform.
Thus $\Omega$ is a uniform domain.

Let us show that the metric $\rho_\Omega(x,y)$ is comparable to the Euclidean distance $|x-y|$ on $\Omega$.
For any two points $x,y\in\Omega$, by considering a compact supported Lipschitz function with Lipschitz constant one that satisfies $f(z)=|z-x|$ in a neighborhood  containing the line segment $x\to y$, we see that
$$|x-y|\le \rho_\Omega(x,y).$$

By the definition of uniform domains, for any $x,y\in \Omega$, there exists
$\gamma_{x,y}$ connecting $x$ to $y$ with $\ell(\gamma_{x,y})\le C|x-y|$. We therefore deduce that $$\rho_\Omega(x,y)\le \int_{\gamma_{x,y}}|\nabla f(\gamma(t))|\,dt\le \ell(\gamma_{x,y})\le
C|x-y|.$$
The above two inequalities show that $\rho_\Omega$ is equivalent to $|\cdot|$.

Since $\Omega$ is uniform, we see that $\Omega$ is inner uniform, i.e., uniform w.r.t. to the metric $\rho_\Omega$.
\end{proof}
\begin{remark}\label{boundary}\rm
As  a consequence of  the fact that the two metrics are equivalent,  we see that the completion of $(\Omega,\rho_\Omega)$ is just $(\overline{\Omega},\rho_\Omega)$. Note that it does not hold in general that the completion of $\Omega$ equals $\overline{\Omega}$, see \cite[Chapter 1.3 \& Chapter 1.4]{gsa11}.
Since $\partial\Omega$ has measure zero, to say that an operator is boundeded on $L^p(\overline{\Omega})$
is equivalent to that bounded on $L^p({\Omega})$.
\end{remark}

In what follows, we use $B_\rho(x,r)$ to denote the metric ball
$$B_\rho(x,r):=\{y\in\Omega:\,\rho_\Omega(x,y)<r\}.$$
The following result follows from \cite{acdh} and \cite[Theorem 1.9]{cjks16}, by using the heat kernel estimates on inner uniform domains from \cite{gsa11}.

\begin{proof}[Proof of Theorem \ref{main-neumann}]
By Proposition \ref{inn-uniform}, the domain $(\Omega,\rho_\Omega)$ is an inner uniform domain.
By \cite[Theorem 3.10]{gsa11}, the heat kernel $p^N_t(x,y)$ of $e^{-t\LV_N}$ satisfies two side Gaussian bounds, i.e.,
$$\frac{c}{t^{n/2}}e^{-\frac{\rho_\Omega(x,y)^2}{ct}}\le p_t^N(x,y)\le \frac{C}{t^{n/2}}e^{-\frac{\rho_\Omega(x,y)^2}{Ct}}$$
for all $t>0$ and $x,y\in\Omega$.

By \cite[Theorem 1.9]{cjks16}, the fact that Riesz operator $\nabla\mathscr{L}_N^{-1/2}$ is bounded on $L^p(\Omega)$, is equivalent to that, for any  ball $B_\rho(x_0,r)$ with $x_0\in {\Omega}$ and any weak solution $u$ of $\mathscr{L}_Nu=0$ in $B_\rho(x_0,3r)$, it holds that
$$\lf(\fint_{B_\rho(x_0,r)}|\nabla u|^{p}\,dx\r)^{1/p}\le \frac {C}r \fint_{B_\rho(x_0,2r)}|u|\,dx.\leqno({RH}_{\rho,p})$$

Note that, $\mathscr{L}_Nu=0$ in $ B_\rho(x_0, r)$ implies that for any $\psi\in W^{1,2}(B_\rho(x_0, r))$, it holds
$$\int_{B_\rho(x_0, r)}A\nabla u\cdot\nabla \psi\,dx=0.$$
Thus we necessarily have  $\partial_\nu u=0$ on $B_\rho(x_0,\alpha_2r)\cap\partial\Omega$
provided the set is not empty.

By the equivalence of $\rho_\Omega$ and the Euclidean distance from Proposition \ref{inn-uniform},
we see that there exist $1<\gamma_1<\gamma_2<\infty$ such that
$$ B(x_0,r/\gamma_1)\cap \Omega\subset B_\rho(x_0,r)\subset B_\rho(x_0,2r)\subset B(x_0,2\gamma_2r)\cap\Omega.$$
The equivalence of $({RH}_{\rho,p})$ and $(RH_p)$ follows, and completes the proof.
\end{proof}

\begin{remark}\label{poincare}\rm
(i) It was proved in \cite[Theorem 3.12]{gsa11} that the Poincar\'e inequality holds on $(\Omega,\rho_\Omega)$, i.e.,
 $$\fint_{B_\rho(x,r)}|f-f_{B_\rho(x,r)}|\,dy\le Cr\left(\fint_{B_\rho(x,r)}|\nabla f|^2\,dy\right)^{1/2},$$
 which together with a doubling (measure) property is equivalent the two side Gaussian bounds of the heat kernel $p^N_t(x,y)$. However, it is easy to see from the figure, the above Poincar\'e inequality may not hold on $\Omega$ with Euclidean metric.
 \begin{figure}[ht]
\centerline{ \epsfig{file=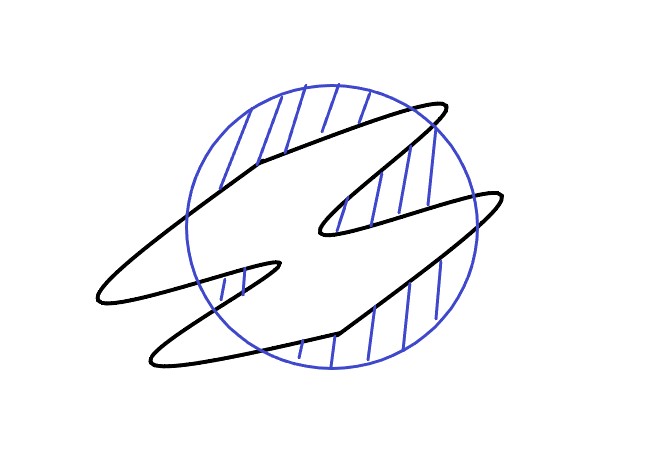, scale=0.5}  }
             \caption{A ball in an exterior smooth domain}
\end{figure}
Instead, from the comparability of the Euclidean metric and intrinsic distance,  one has a weak Poincar\'e inequality for the Euclidean metric, i.e., there exists $\lambda>1$, such that
$$\fint_{B(x,r)\cap\Omega}|f-f_{B(x,r)\cap\Omega}|\,dy\le Cr\left(\fint_{\lambda B(x,r)\cap\Omega}|\nabla f|^2\,dy\right)^{1/2}.$$

(ii) By \cite[Theorem 4]{at01} and \cite[Theorem 1.1]{ge12}, Theorem \ref{main-neumann} holds for the case of $\Omega$ being bounded.
The approach perhaps is well-known, we sketch a proof for completeness.
It follows from \cite[Theorem 4]{at01} that $\|\mathscr{L}_N^{1/2}f\|_{L^p(\Omega)}\le C\|\nabla f\|_{L^p(\Omega)}$ for all $1<p<\infty$. Moreover, it follows from \cite[Theorem 1.1]{ge12} that,
for $p>2$, the $L^p$-boundedness of the operator
$\nabla\mathscr{L}_N^{-1}\mathrm{div}$ is equivalent to that $(RH_p)$
holds on $\Omega$.
If $\nabla \mathscr{L}_N^{-1/2}$ is bounded on $L^p(\Omega)$, $p>2$, then by a duality argument and  the $L^{p'}$-boundedness of
$\nabla \mathscr{L}_N^{-1/2}$, we see that $\nabla\mathscr{L}_N^{-1}\mathrm{div}$ is bounded on $L^p(\Omega)$, which implies $(RH_p)$. Conversely, $(RH_p)$ implies $L^{p}$-boundedness $\nabla\mathscr{L}_N^{-1}\mathrm{div}$, and then a duality argument yields the $L^{p'}$-boundedness $\nabla\mathscr{L}_N^{-1}\mathrm{div}$.  Finally \cite[Theorem 4]{at01} implies that
\begin{eqnarray*}
\|\LV^{-1/2}_N \mathrm{div}\|_{L^{p'}(\Omega)\to L^{p'}(\Omega)}=\|\LV_N^{1/2}\LV_N^{-1} \mathrm{div}\|_{L^{p'}(\Omega)}\le
\|\nabla\LV^{-1}_N \mathrm{div}\|_{L^{p'}(\Omega)}\le C.
\end{eqnarray*}
Once again, a duality argument implies that
\begin{eqnarray*}
\|\nabla \LV^{-1/2}_N \|_{L^{p}(\Omega)\to L^{p}(\Omega)}\le C.
\end{eqnarray*}
\end{remark}

\subsection{Neumann operators with VMO coefficients}
\hskip\parindent  In this part, we prove Theorem \ref{app-neumann}.
Let us begin with some basic estimates for the Neumann operator on
exterior Lipschitz and $C^1$ domains; see \cite{agg97} for related results on exterior
$C^{1,1}$ domains.

 Recall that $A\in VMO(\rn)$, and
 $$p_\LV:=\sup\{p>2: \,\nabla \LV^{-1/2}\ \text{is bounded on}\, L^p(\rn)\}.$$
Since $\nabla\LV^{-1/2}$ is always bounded on $L^p(\rn)$ for $1<p\le 2$ (cf. \cite{cd99,Si}),
the operator $\nabla\LV^{-1}\text{div}$ is bounded on $L^p(\rn)$ for all $p_\LV'<p<p_\LV$.

\begin{prop}\label{reg-neu-har}
Let $\Omega\subset \rn$ be an exterior Lipschitz domain, $n\ge 2$.

(i) For $g\in W^{-1/2,2}(\partial\Omega)$ satisfying the compatibility condition $\int_{\partial\Omega}g\,d\nu=0$, then there exists a unique (up to module constants) $u\in \dot{W}^{1,2}(\Omega)$
such that $\LV_N u=0$ in $\Omega$ and $\partial_\nu u=g$ on $\partial\Omega$, and
$$\|u\|_{\dot{W}^{1,2}(\Omega)}\le C\|g\|_{W^{-1/2,2}(\Omega)}.$$

(ii) There exists $\epsilon>0$ such that for $2<p<\min\{p_\LV,3+\epsilon\}$ when $n\ge 3$, $2<p<\min\{p_\LV,4+\epsilon\}$ when $n=2$,
and $g\in W^{-1/p,p}(\partial\Omega)$ satisfying the compatibility condition $\int_{\partial\Omega}g\,d\nu=0$, the solution $u$ further satisfies
$$\|u\|_{\dot{W}^{1,p}(\Omega)}\le C\|g\|_{W^{-1/p,p}(\Omega)}.$$

(iii) If  $\partial \Omega$ is in $C^1$ class, then the conclusion of (ii) holds for all $2<p<p_\LV$.
\end{prop}
\begin{proof}
(i) follows from the Lax-Milgram theorem, let us prove (ii).

For $p>2$, $g\in W^{-1/p,p}(\partial\Omega)$ belongs to $W^{-1/2,2}(\partial\Omega)$. Let $u$
be the solution found in (i).
Choose a large enough ball $B(0,R)$ such that $\Omega^c\subset B(0,R-1)$. Let $\psi\in C^\infty(\rn)$
that satisfies $\psi=1$ on $\rn\setminus B(0,R+1)$ and $\supp \psi\subset \rn\setminus B(0,R)$.
Then $u\psi$ satisfies that for any $\phi\in \dot{W}^{1,2}(\rn)$
\begin{equation}
\int_\rn A\nabla(u\psi)\cdot\nabla\phi\,dx=\int_\rn \left[u A\nabla\psi\cdot\nabla\phi+\phi A\nabla u\cdot\nabla\psi\right] \,dx.
\end{equation}
Then the regularity result on $\rn$ (cf. \cite[Proposition 2.3]{shz05}) together with the boundedness of the
Riesz transform $\nabla\LV^{-1/2}$ implies that for $2<p<p_\LV$,
\begin{eqnarray*}
\|\nabla(u\psi)\|_{L^p(\rn)}&&\le C\|u\nabla \psi\|_{L^p(\rn)}+C\|\nabla u\cdot\nabla\psi\|_{L^{p_\ast}(\rn)}\\
&&\le C\|u\|_{L^p(B(0,R+1)\setminus B(0,R))}+C\|\nabla u\|_{L^{p_\ast}(B(0,R+1)\setminus B(0,R))}\\
&&\le C\|u\|_{L^1(B(0,R+1)\setminus B(0,R))}+C\|\nabla u\|_{L^{p_\ast}(B(0,R+1)\setminus B(0,R))},
\end{eqnarray*}
where $p_\ast=\frac{np}{n+p}$, and the last step follows from the Poincar\'e inequality on the ring
$B(0,R+1)\setminus B(0,R)$. Up to iterating the arguments several times, we see that
\begin{eqnarray*}
\|\nabla u\|_{L^p(\rn\setminus B(0,R+1))}&&\le C\|u\|_{L^2(B(0,R+1)\setminus B(0,R))}+C\|\nabla u\|_{L^{2}(B(0,R+1)\setminus B(0,R))}\\
&&\le C\|g\|_{W^{-1/2,2}(\partial\Omega)}\\
&&\le C\|g\|_{W^{-1/p,p}(\partial\Omega)}.
\end{eqnarray*}

For $u$ on $B(0,R+2)\cap\Omega$, let us consider the dual equation
\begin{equation}\label{lip-neumann}
\begin{cases}
\LV_N v=-\mathrm{div} f & \text{in}\, B(0,R+2)\cap\Omega,\\
\nu\cdot A\nabla v=-\nu\cdot f & \text{on} \, \partial\Omega\cup \partial B(0,R+2),
\end{cases}
\end{equation}
where $f\in L^{p'}(B(0,R+2)\cap\Omega)$, where $p'$ is the H\"older conjugate of $p$.
Since
the domain $B(0,R+2)\cap\Omega$ is a bounded Lipschitz domain and $A\in VMO(\rn)$,
by \cite[Theorem 1.1]{ge12}, for $2<p<3+\epsilon$ when $n\ge 3$, $2<p<4+\epsilon$ when $n=2$,
there exists a unique (up to modulo constants) solution to the equation
that satisfies
\begin{equation}\label{lip-neu-reg}
\|\nabla v\|_{L^{p'}(B(0,R+2)\cap\Omega)}\le C\|f\|_{L^{p'}(B(0,R+2)\cap\Omega)},
\end{equation}

Let $\tilde\psi\in C^\infty_c(B(0,R+2))$ be such that $\psi=1$ on $B(0,R+1)$.
By the previous estimate, we  conclude that
\begin{eqnarray*}
\int_{B(0,R+2)\cap\Omega}f\cdot \nabla u\,dx&&=\int_{B(0,R+2)\cap\Omega} A\nabla v\cdot \nabla u\,dx=
\int_{B(0,R+2)\cap\Omega} \nabla [(v\psi)+v(1-\psi)]\cdot A\nabla u\,dx\\
&&=\int_{\partial \Omega} v g\,d\sigma+\int_{B(0,R+2)\cap\Omega} \nabla [v(1-\psi)]\cdot A\nabla u\,dx\\
&&\le C\|v\|_{W^{1/p,p'}(\partial\Omega)}\|g\|_{W^{-1/p,p}(\Omega)}+C\|v\|_{W^{1,p'}(B(0,R+2))}\|\nabla u\|_{L^p(B(0,R+2)\setminus B(0,R+1))}\\
&&\le C\|f\|_{L^{p'}(\Omega\cap B(0,R+2))}\|g\|_{W^{-1/p,p}(\Omega)},
\end{eqnarray*}
which implies that
\begin{eqnarray*}
\|\nabla u\|_{L^{p}(\Omega\cap B(0,R+2))}\le C\|g\|_{W^{-1/p,p}(\Omega)}.
\end{eqnarray*}
We therefore conclude that
\begin{eqnarray*}
\|\nabla u\|_{L^{p}(\Omega)}\le C\|g\|_{W^{-1/p,p}(\Omega)},
\end{eqnarray*}
where $2<p<\min\{p_\LV,3+\epsilon\}$ when $n\ge 3$, $2<p<\min\{p_\LV,4+\epsilon\}$ when $n=2$.

(iii) The proof of $C^1$ boundary case is the same as of (ii), by noting that on bounded
$C^1$ domains, the estimate \eqref{lip-neu-reg} for the problem \eqref{lip-neumann} holds for
 $\epsilon=\infty$ (cf. \cite{aq02}).
\end{proof}

\begin{prop}\label{reg-neu-poi}
Let $\Omega\subset \rn$ be an exterior Lipschitz domain, $n\ge 2$. Let $A\in VMO(\rn)$.

(i) There exists $\epsilon>0$ such that for $\max\{p_\LV',\frac {3+\epsilon}{2+\epsilon}\}<p<\min\{p_\LV,3+\epsilon\}$ when $n\ge 3$, $\max\{p_\LV',\frac {4+\epsilon}{3+\epsilon}\}<p<\min\{p_\LV, 4+\epsilon\}$ when $n=2$,
and $g\in C^\infty_c(\Omega,\rn)$, there exists a unique solution $u\in \dot{W}^{1,p}(\Omega)\cap \dot{W}^{1,2}(\Omega)$ to the problem
\begin{equation}\label{equation-neu}
\begin{cases}
\LV_N u =-\mathrm{div}\, g & in \, \Omega,\\
\nu\cdot A\nabla u=\nu\cdot g & on \, \partial\Omega,
\end{cases}
\end{equation}
that satisfies
$$\|\nabla u\|_{L^p(\Omega)}\le C\|g\|_{L^p(\Omega)}.$$
(ii) If  $\partial \Omega$ is in $C^1$ class, then the conclusion of (i) holds for $\epsilon=\infty$.
\end{prop}
\begin{proof}
Let us prove (i), the proof of (ii) is the same by using corresponding regularity result on bounded $C^1$ domains (cf. \cite{aq02}).

The case $p=2$ follows from the Lax-Milgram theorem. Let us consider $p>2$.
Since $g\in C^\infty_c(\Omega,\rn)$, we may simply view that  $g\in C^\infty_c(\rn,\rn)$ with
$\supp g\subset \Omega$. For $2<p<p_\LV$, we let $v$ be the solution in $\dot{W}^{1,p}(\rn)\cap \dot{W}^{1,2}(\Omega)$ such that
$\LV v=-\text{div}g$ in $\rn$. Then it holds that
$$\|\nabla v\|_{L^p(\rn)}=\|\nabla\LV^{-1}\text{div}g\|_{L^p(\rn)}\le C\|g\|_{L^p(\rn)}.$$
Since $v\in \dot{W}^{1,p}(\rn)$, $\partial_\nu v\in W^{-1/p,p}(\partial\Omega)$. Moreover,
$$\int_{\partial\Omega} \partial_\nu v\,d\sigma=\int_{\rn\setminus\Omega} [\text{div} (A\nabla v)+A\nabla v\cdot\nabla 1]\,dx=0,$$
$\partial_\nu v$ satisfies also the compatibility condition.

By Proposition \ref{reg-neu-har}, there exists a unique $w\in \dot{W}^{1,p}(\Omega)$
such that $\LV_N w=0$ in $\Omega$, $\partial_\nu w=-\partial_\nu v$ on $\partial\Omega$ and
$$\|\nabla w\|_{L^p(\Omega)}\le C\|\partial_\nu v\|_{W^{-1/p,p}(\partial\Omega)}\le C\|\nabla v\|_{\dot{W}^{1,p}(\rn)}\le C\|g\|_{L^p(\Omega)}. $$
Note that since $\partial_\nu v\in W^{-1/p,p}(\partial\Omega)\subset W^{-1/2,2}(\partial\Omega)$,
$w$ also belongs to $\dot{W}^{1,2}(\Omega)$.
Finally $u=w+v$ is the required solution to \eqref{equation-neu}.

For $p<2$, let $u,\,v\in \dot{W}^{1,2}(\Omega)$ be the two solutions to \eqref{equation-neu}
with $-\LV_N u=\text{div}f$, $-\LV_N v=\text{div}g$, where $f,\,g\in C^\infty_c(\Omega,\rn)$.
Then it holds that
\begin{eqnarray*}
\int_\Omega \nabla u\cdot g\,dx&&=\int_\Omega u\LV_N v\,dx=\int_\Omega \nabla u\cdot A\nabla v\,dx=\int_\Omega f\cdot \nabla v\,dx\\
&&\le \|f\|_{L^{p}(\Omega)}\|\nabla v\|_{L^{p'}(\Omega)}\\
&&\le C\|f\|_{L^{p}(\Omega)}\|g\|_{L^{p'}(\Omega)}.
\end{eqnarray*}
Taking supremum over $g$ implies that
\begin{eqnarray*}
\| \nabla u\|_{L^{p}(\Omega)}\le C\|f\|_{L^{p}(\Omega)},
\end{eqnarray*}
as desired.
\end{proof}

We can now give the proof of Theorem \ref{app-neumann}.
\begin{proof}[Proof of Theorem \ref{app-neumann}]
Since the heat kernel of $e^{-t\LV_N}$ satisfies two side Gaussian bounds, the Riesz transform $\nabla \LV_N^{-1/2}$ is bounded on $L^p(\Omega)$ for $1<p<2$ by \cite{cd99}; see also \cite{Si}.

For $2<p<\infty$, by Theorem \ref{main-neumann}, it suffices to establish $(RH_p)$ for harmonic functions.
Suppose that $u$ is a harmonic function in $B(x_0,3r)\cap\Omega$, which satisfies additionally $\partial_\nu u=0$ on
$B(x_0,3r)\cap\partial\Omega$ if the set is not empty.

(i) Let us first prove the Lipschitz case.  Let $2<p<p_\LV$.
If $B(x_0,3r)\cap\Omega=\emptyset$, then by the fact that $\nabla\LV^{-1/2}$ is bounded on $L^p(\rn)$ is
equivalent to $(RH_p)$ on $\rn$ (cf. \cite[Theorem 1.1]{shz05} or \cite{cjks16}), we see that
$$\lf(\fint_{B(x_0,r)}|\nabla u|^{p}\,dx\r)^{1/p}\le
\frac{C}{r}\fint_{B(x_0,2r)}|u|\,dx. \leqno({RH}_p)$$

Suppose now $B(x_0,3r)\cap\Omega\neq\emptyset$. \cite[Lemma 4.1]{ge12} established
$({RH}_p)$ for  $2<p<3+\epsilon$ when $n\ge 3$, $ 2<p<4+\epsilon$  when $n=2$,
for some $\epsilon>0$, on bounded Lipschitz domains. In particular, \cite[Lemma 4.1]{ge12} shows that,
if $x_0\in \partial\Omega$ and $r<r_0$ for some fixed $r_0>0$, the $({RH}_p)$ holds for all these
$p$.

For those balls that satisfy  $B(x_0,3r)\cap\Omega\neq\emptyset$, $x_0\notin\partial\Omega$ and $r<r_0$,
we divide the ball $B(x_0,r)$ into the union of a sequence of small balls $\{B(x_j,r/12)\}_{1\le j\le c(n)}$,
where $c(n)$ only depends on the dimension. For each $1\le j\le c(n)$, if $B(x_j,r/4)\cap\partial \Omega=\emptyset$, then the interior estimate applies and it holds that
$$\lf(\fint_{B(x_j,r/12)}|\nabla u|^{p}\,dx\r)^{1/p}\le
\frac{C}{r}\fint_{B(x_j,r/6)}|u|\,dx\le \frac{C}{r}\fint_{B(x_0,2r)}|u|\,dx. $$
Otherwise, we can find $\tilde x_j\in B(x_j,r/4)\cap\partial \Omega$ so that
$$B(x_j,r/12)\subset B(\tilde x_j,r/3)\subset B(\tilde x_j,r)\subset B(x_0,2r).$$
We therefore conclude that
\begin{eqnarray*}
  \lf(\fint_{B(x_j,r/12)}|\nabla u|^{p}\,dx\r)^{1/p}&&\le C\lf(\fint_{B(\tilde x_j,r/3)}|\nabla u|^{p}\,dx\r)^{1/p}\\
  &&\le \frac{C}{r}\fint_{B(\tilde x_j,2r/3)}|u|\,dx\\
  &&\le \frac{C}{r} \fint_{B(x_0,2r)}|u|\,dx.
\end{eqnarray*}
A summation over $1\le j\le c(n)$ shows that
\begin{eqnarray*}
  \lf(\fint_{B(x_0,r)}|\nabla u|^{p}\,dx\r)^{1/p}&&\le \frac{C}{r} \fint_{B(x_0,2r)}|u|\,dx,
\end{eqnarray*}
for any ball $B(x_0,r)$ with $B(x_0,3r)\cap\partial\Omega\neq \emptyset$ and $r<r_0$.

Let $R_0=4\mathrm{diam}(\rn\setminus\Omega)$. For balls $B(x_0,r)$ with $B(x_0,3r)\cap\partial\Omega\neq \emptyset$ and $r_0\le r<2R_0$, we simply divide the ball into the union of small balls $\{B(x_j,r_0/2)\}_{1\le j\le c(n,R_0/r_0)}$, where $c(n,R_0/r_0)$ depends only on the dimension and the ratio $R_0/r_0$, and repeat the above covering argument to see that
\begin{eqnarray*}
  \lf(\fint_{B(x_0,r)}|\nabla u|^{p}\,dx\r)^{1/p}&&\le \frac{C}{r} \fint_{B(x_0,2r)}| u|\,dx,
\end{eqnarray*}
where $2<p<3+\epsilon$ when $n\ge 3$, $ 2<p<4+\epsilon$  when $n=2$.

Suppose now $B(x_0,3r)\cap\partial\Omega\neq \emptyset$ and $r\ge 2R_0$. If $B(x_0,5r/4)\cap\partial\Omega= \emptyset$, then by applying the interior estimate, we deduce that
\begin{eqnarray*}
  \lf(\fint_{B(x_0,r)}|\nabla u|^{p}\,dx\r)^{1/p}&&\le \frac{C}{r} \fint_{B(x_0,5r/4)}|u|\,dx
  \le \frac{C}{r} \fint_{B(x_0,2r)}|u|\,dx,
\end{eqnarray*}
for $2<p<p_\LV$.

If $B(x_0,5r/4)\cap\partial\Omega\neq\emptyset$, then by the fact $r\ge 2R_0=8\mathrm{diam}(\rn\setminus\Omega)$, we have
$$\rn\setminus\Omega\subset B(x_0,3r/2).$$

Let $\phi$ be a bump function with support in $B(x_0,2r)$ such that $\phi=1$ on $B(x_0,3r/2)$ and $|\nabla\phi|\le C/r$.
Since $\LV_N u=0$ in $B(x_0,3r)\cap\Omega$ and $\partial_\nu u =0$ on $\partial\Omega$, we have that
for any $\psi\in \dot{W}^{1,2}(\Omega)\cap \dot{W}^{1,p'}(\Omega)$, it holds that
\begin{eqnarray*}
\int_\Omega A\nabla(u\phi^2)\cdot\nabla \psi\,dx&&= \int_\Omega A\nabla u \cdot\nabla (\psi\phi^2)\,dx-
\int_\Omega 2\phi\psi A\nabla u \cdot \nabla \phi\,dx+\int_\Omega 2u\phi A\nabla \phi\cdot\nabla \psi\,dx\\
&&=-
\int_\Omega 2\psi\phi A\nabla u \cdot \nabla \phi\,dx+\int_\Omega 2u\phi A\nabla \phi\cdot\nabla \psi\,dx.
\end{eqnarray*}

Since $\Omega$ is Lipschitz, $\Omega\cap B(0,R_0)$ is also Lipschitz and hence uniform. By Jones \cite[Theorem 1.2]{jo81} (see also \cite{hk91}), for any $1<q<\infty$, there exists an extension operator $E$
such that $Ef\in \dot{W}^{1,q}(\rn)$, $Ef=f$ on $\Omega\cap B(0,R_0)$ and it holds
$$\|\nabla Ef\|_{L^{q}(\rn)}\le C\|\nabla f\|_{L^{q}(\Omega\cap B(0,R_0)}.$$
For Sobolev function $w$ on $\Omega$, we let $\tilde w$ be such that $\tilde w=w$ on $\Omega$
and $\tilde w=Ew$ on $\rn\setminus \Omega$ below.

By using this extension operator, we further deduce via the Poincar\'e inequality that
\begin{eqnarray*}
\int_\Omega A\nabla((u-c)\phi^2)\cdot\nabla \psi\,dx&&=-
\int_\Omega 2\psi\phi A\nabla u \cdot \nabla \phi\,dx+\int_\Omega 2(u-c)\phi A\nabla \phi\cdot\nabla \psi\,dx\\
&&\le C\|\nabla u\cdot\nabla\phi\|_{L^{p_\ast}(\Omega)}\|\tilde\psi\|_{L^{(p')^\ast}(\rn)}+
\frac{C}{r}\|\nabla\psi\|_{L^{p'}(\Omega)}\|\tilde u-c\|_{L^p(B(x_0,2r))}\\
&&\le \frac{C}{r}\|\nabla \psi\|_{L^{p'}(\Omega)} \|\nabla u\|_{L^{p_\ast}(B(x_0,2r)\cap\Omega)},
\end{eqnarray*}
where $c=\tilde u_{B(x_0,2r)}$, and $B(0,R_0)\subset B(x_0,2r)$ since $r\ge 2R_0$ and
$\rn\setminus\Omega\subset B(x_0,3r/2)$.

For any $g\in C^\infty_c(\Omega,\rn)$, by Proposition \ref{reg-neu-poi},
for $\max\{p_\LV',\frac {3+\epsilon}{2+\epsilon}\}<p'<2$ when $n\ge 3$, $\max\{p_\LV',\frac {4+\epsilon}{3+\epsilon}\}<p'<2$ when $n=2$,
there exists a unique solution $v\in \dot{W}^{1,p'}(\Omega)\cap \dot{W}^{1,2}(\Omega)$ to the problem
\eqref{equation-neu}, and that
\begin{eqnarray*}
\int_\Omega A\nabla((u-c)\phi^2)\cdot g\,dx&&=\int_\Omega A\nabla((u-c)\phi^2)\cdot\nabla v\,dx\\
&&\le \frac{C}{r}\|\nabla v\|_{L^{p'}(\Omega)} \|\nabla u\|_{L^{p_\ast}(B(x_0,2r)\cap\Omega)}\\
&&\le   \frac{C}{r}\|g\|_{L^{p'}(\Omega)} \|\nabla u\|_{L^{p_\ast}(B(x_0,2r)\cap\Omega)}.
\end{eqnarray*}
Taking supremum over $g$ implies that
\begin{eqnarray*}
\|\nabla u\|_{L^{p}(B(x_0,r)\cap\Omega)}\le \frac{C}{r}\|\nabla u\|_{L^{p_\ast}(B(x_0,2r)\cap\Omega)}.
\end{eqnarray*}
Up to iterating this argument several times, we arrive at
\begin{eqnarray*}
\|\nabla u\|_{L^{p}(B(x_0,r)\cap\Omega)}\le \frac{C}{r}\|\nabla u\|_{L^{2}(B(x_0,2r)\cap\Omega)},
\end{eqnarray*}
which implies that
\begin{eqnarray*}
\left(\fint_{B(x_0,r)\cap\Omega}|\nabla u|^p\,dx\right)^{1/p}\le C\left(\fint_{B(x_0,2r)\cap\Omega}|\nabla u|^2\,dx\right)^{1/2}\le \frac{C}{r}\fint_{B(x_0,3r)}|u|\,dx.
\end{eqnarray*}

Theorem \ref{main-neumann} then yields that the Riesz transform is bounded on
$L^p(\Omega)$ for $2<p<\min\{p_\LV,3+\epsilon\}$ when $n\ge 3$, $2<p<\min\{p_\LV, 4+\epsilon\}$ when $n=2$, and therefore is bounded on
$L^p(\Omega)$ for $1<p<\min\{p_\LV,3+\epsilon\}$ when $n\ge 3$, $1<p<\min\{p_\LV, 4+\epsilon\}$ when $n=2$.

By a duality argument, for any $f,g\in C^\infty_c(\Omega)$, we have
$$<f,g>=<\nabla\LV_N^{-1/2}f,\nabla\LV_N^{-1/2}g>\le C\|f\|_{L^p(\Omega)}\|\nabla \LV^{-1/2}_Ng\|_{L^{p'}(\Omega)},$$
which implies that
for $\max\{p_\LV',\frac {3+\epsilon}{2+\epsilon}\}<p'<\infty$ when $n\ge 3$, $\max\{p_\LV',\frac {4+\epsilon}{3+\epsilon}\}<p'<\infty$ when $n=2$,
it holds
$$\|g\|_{L^{p'}(\Omega)}\le C\|\nabla \LV^{-1/2}_Ng\|_{L^{p'}(\Omega)}.$$

Therefore, it holds for each $\max\{p_\LV',\frac {3+\epsilon}{2+\epsilon}\}<p'<\min\{p_\LV,3+\epsilon\}$ when $n\ge 3$, $\max\{p_\LV',\frac {4+\epsilon}{3+\epsilon}\}<p'<\min\{p_\LV, 4+\epsilon\}$ when $n=2$,
 that
$$\|\nabla f\|_{L^p(\Omega)}\sim \|\LV_N^{1/2}f\|_{L^{p}(\Omega)}.$$
which completes the proof of (i).

(ii) For the case of $C^1$ domains, note that the estimates \cite[(4.4)]{ge12} holds for each $\eta\in (0,1)$ on $C^1$ domains (cf. \cite[Remark 4.5]{shz05}). Using this fact, the same proof of \cite[Lemma 4.1]{ge12} shows that,
if $x_0\in \partial\Omega$ and $r<r_0$ for some fixed $r_0>0$, $({RH}_p)$ holds for all $2<p<\infty$ on these boundary balls. The rest proof is the same as the Lipschitz case.
\end{proof}

\section{Additional results and remarks}
\hskip\parindent In this part, we provide some more concrete examples where one can take $p_\LV=\infty$,
and apply our main results to the inhomogeneous Dirichlet/Neumann problem.

While the assumption that $A$ belongs to VMO space suffices for the local behavior of harmonic functions  by the theory of Caffarelli-Peral \cite{cp98} (see also \cite{aq02,shz05,ge12}), more restrictions are needed for the large scales (cf. \cite[Proposition 1.1]{jl20}).

Note that in the following corollary one of course can replace the identity matrix by any symmetric and positive definite constant matrix. In particular, when $A=I_{n\times n}$, the following results settle the cases of the Dirichlet
and Neumann Laplacian, $\Delta_D$ and $\Delta_N$.
\begin{corollary}\label{app-inf-dirichlet}
Let $\Omega\subset \rn$ be an exterior  Lipschitz domain, $n\ge 3$.
Let $A\in VMO(\rn)$ satisfy
$$\fint_{B(x_0,r)}|A-I_{n\times n}|\,dx\le \frac{C}{r^\delta}$$
for some $\delta>0$, all $r>1$ and all $x_0\in\rn$.

(i) There exists $\epsilon>0$ such that for each $1<p<\min\{n,3+\epsilon\}$, there exists $C>1$
such that for all $f\in \dot{W}^{1,p}_0(\Omega)$
$$C^{-1}\|\nabla f\|_{L^p(\Omega)}\le \|\LV_D^{1/2}f\|_{L^p(\Omega)}\le C\|\nabla f\|_{L^p(\Omega)}.$$

(ii) If $\Omega$ is $C^1$, then the conclusion of $(i)$ holds for all $1<p<n$.
\end{corollary}

To prove Corollary \ref{app-inf-dirichlet}, let us recall the following result,
 which is a special case of \cite[Theorem 3.1]{jl20}, i.e., $w=w_0\equiv 1$ there.
\begin{thm}\label{main-denegerate}
Let $A,A_0$ be $n\times n$ matrixes that satisfy  uniformly elliptic  conditions.
Suppose there exists $\epsilon>0$ such that
$$\fint_{B(y,r)} |A-A_0|\,dx\le \frac{C}{r^\epsilon},\quad\, \forall \,y\in\rn \, \& \,  r>1.$$
Let $\L=-\mathrm{div}(A\nabla)$ and $\L_0=-\mathrm{div}(A_0\nabla).$
Then if $\nabla\mathcal L_0^{-1/2}$ and $\nabla(1+\mathcal L)^{-1/2}$  are bounded on $L^p(\rn)$ for some $p\in (2,\infty)$,
$\nabla\mathcal L^{-1/2}$ is bounded on $L^p(\rn)$.
\end{thm}

The proof of the following result is taken from \cite{cjks16}, we include the proof for completeness.
\begin{prop}\label{local-riesz}
Let $A\in VMO(\rn)$. The local Riesz transform $\nabla (1+\LV)^{-1/2}$ is bounded on $L^p(\rn)$, for all $1<p<\infty$.
\end{prop}
\begin{proof}
The case $1<p\le 2$ follows from \cite[Theorem 1.2]{cd99}. Let us prove the case $2<p<\infty$.

  For the operator $\LV=-\mathrm{div}A\nabla$, with $A\in VMO(\rn)$, it follows from the perturbation theory of
Caffarelli and Peral \cite{cp98} that, for each $2<p<\infty$, there exists $r_0>0$ such that
for any
$\LV$-harmonic functions $u$ on  balls  $B(x_0,2r)$ with $r<r_0$, it holds that
 $$\left(\fint_{B(x_0,r)}|\nabla u|^p\,dx\right)^{1/p}\le \frac{C}{r}\fint_{B(x_0,2r)}|u|\,dx;$$
 see \cite[Lemma 4.7 \& Theorem 4.13]{shz05}. This implies that for $\LV v=g$ in $B(x_0,2r)$, $g\in L^\infty(B(x_0,2r))$, $r<r_0$, that
\begin{equation}\label{est-poisson}
\left(\fint_{B(x_0,r)}|\nabla v|^{p}\,d\mu\right)^{1/p}\leq \frac {C}{r}\left(\fint_{B(x_0,2r)}|v|\,d\mu+
r^2\left(\fint_{B(x_0,2r)}|g|^{p}\,d\mu\right)^{1/p}\right),
\end{equation}
see \cite[Theorem 3.6]{cjks16}.

Note that for each time $t>0$ and $y\in \rn$, the heat kernel satisfies
$$\LV p_t(\cdot,y)=-\partial_t p_t(\cdot,y),$$
where $\partial_t p_t(\cdot,y)$ is controlled by
\begin{equation}\label{est-time-heat}
|\partial_t p_t(\cdot,y)|\le \frac{C(n)}{t^{1+n/2}}e^{-\frac{|x-y|^2}{ct}}.
\end{equation}
For $t< r_0^2$, we decompose $\rn$ into $B=B(y,\sqrt t)$ and the sets $U_k(B):=B(y,2^{k}\sqrt  t)\setminus B(y,2^{k-1}\sqrt t),$  $k\ge 1$.
By \eqref{est-poisson}, we see that
\begin{eqnarray*}
&& \||\nabla_x p_t(\cdot,y)|\|_{L^{p}(B)}\le \frac {Ct^{\frac n{2p}}}{\sqrt t}\left(\fint_{B(y,2\sqrt t)}|p_t(x,y)|\,dx+
t\left(\fint_{B(y,4\sqrt t)}\left|\frac{\partial }{\partial t}p_t(x,y)\right|^{{p}}\,dx\right)^{1/{p}}\right)
\le Ct^{\frac{n}{2}(\frac 1{p}-1)-\frac 12}.
\end{eqnarray*}
Let $\{B_{k,j}=B(x_{k,j},\sqrt t/2)\}_j$ be a maximal  set of pairwise disjoint
balls with radius $2^{-1}\sqrt t$ in $B(y,2^{k+1}\sqrt  t)$. It holds then
 $$B(y,2^{k+1}\sqrt t)\subset \cup_{j}2B_{k,j}$$
 and
 $$\sum_{j}\chi_{2B_{k,j}}(x)\le C(n).$$
 Therefore, by applying \eqref{est-poisson}, \eqref{dirichlet-gaussian} and \eqref{est-time-heat}, we get
 \begin{eqnarray*}
 &&\int_{U_k(B)}|\nabla_x p_t(x,y)|^{p}\,dx\\
 &&\quad \le \sum_{j:\,2B_{k,j}\cap U_k(B)\neq\emptyset}  \int_{2B_{k,j}}|\nabla_x p_t(x,y)|^{p}\,dx \\
 &&\quad\le \sum_{j:\,2B_{k,j}\cap U_k(B)\neq\emptyset} Ct^{\frac n2-\frac p2}
 \left(\fint_{4B_{k,j}}|p_t(x,y)|\,dx+
t\left(\fint_{4B_{k,j}}\left|\frac{\partial }{\partial t}p_t(x,y)\right|^{{p}}\,dx\right)^{1/{p}}\right)^{p}\\
&&\quad\le \sum_{j:\,2B_{k,j}\cap U_k(B)\neq\emptyset}Ct^{\frac n2-\frac p2-\frac{np}{2}}\exp\left\{\frac{-c2^{2k}t}{t}\right\}\\
&&\quad\le \sum_{j:\,2B_{k,j}\cap U_k(B)\neq\emptyset}C|B_{k,j}|t^{-\frac p2-\frac{np}{2}}\exp\left\{-c2^{2k}\right\}\\
&&\quad\le C2^{kn}t^{\frac n2-\frac p2-\frac{np}{2}}\exp\left\{-c2^{2k}\right\}\\
&&\quad\le Ct^{\frac n2-\frac p2-\frac{np}{2}}\exp\left\{-c2^{2k}\right\}.
 \end{eqnarray*}
From this and the estimate of $\|\nabla_x p_t(\cdot,y)\|_{L^{p}(B)}$, we deduce that there exists $\gz>0$ such that for $t<r_0^2$,
 \begin{eqnarray}\label{est-heat}
 &&\int_{\rn}|\nabla_x p_t(x,y)|^{p}\exp\left\{\gz |x-y|^2/t\right\}\,dx\le Ct^{\frac n2-\frac p2-\frac{np}{2}}.
 \end{eqnarray}
The H\"older inequality then implies that for $t<r_0^2$ and $f\in L^p(\rn)$,
\begin{eqnarray*}
&&\|\nabla e^{-t\LV}f\|_{L^p(\rn)}\\
&&\quad\le \left(\int_{\rn}
\left|\int_{\rn}|\nabla_x p_t(x,y)||f(y)|\,dy\right|^p\,dx\right)^{1/p}\\
&&\quad \le \left(\int_{\rn}
\int_{\rn}|\nabla_x p_t(x,y)|^p|f(y)|^p\exp\left\{\frac{\gz |x-y|^2}{t}\right\}\,dy\left(\int_{\rn}\exp\left\{-\frac{c|x-y|^2}{t}\right\}\,dy \right)^{p-1}\,dx\right)^{1/p}
\\
&&\quad\le \left(\int_{\rn}Ct^{\frac n2-\frac p2-\frac{np}{2}+\frac{n(p-1)}{2}}|f(y)|^p    \right)^{1/p}\\
&&\quad\le \frac{C}{t^{p/2}}\|f\|_{L^p(\rn)},
\end{eqnarray*}
i.e.,
$$\|\nabla e^{-t\LV}\|_{L^p(\rn)\to L^p(\rn)}\le \frac{C}{\sqrt{t}}$$
for $0<t<r_0^2$. This implies that for $t\ge r_0^2$,
$$\|\nabla e^{-t\LV}\|_{L^p(\rn)\to L^p(\rn)}\le \|\nabla e^{-r_0^2\LV}\|_{L^p(\rn)\to L^p(\rn)}\| e^{-(t-r_0^2)\LV}\|_{L^p(\rn)\to L^p(\rn)}\le C.$$
By \cite[Theorem 1.5]{acdh}, we see that the local Riesz transform
 $\nabla (1+\LV)^{-1/2}$ is bounded on $L^p(\rn)$, for all $2<p<\infty$.
\end{proof}

\begin{proof}[Proof of Corollary \ref{app-inf-dirichlet}]
By the assumption $A\in VMO(\rn)$, by Proposition \ref{local-riesz} we see that $\nabla(1+\LV)^{-1/2}$ is bounded on $L^p(\rn)$
for all $2<p<\infty$.
By this, and the assumption that
$$\fint_{B(y,r)} |A-I_{n\times n}|\,dx\le \frac{C}{r^\epsilon},\quad\, \forall \,y\in\rn \, \& \,  r>1,$$
we conclude from Theorem \ref{main-denegerate} that $\nabla\LV^{-1/2}$ is bounded on $L^p(\rn)$ for all
$2<p<\infty$. Thus $p_\LV=\infty$.
The desired result then follows from Theorem \ref{app-dirichlet}.
\end{proof}

The case of Neumann operators follows similarly.
\begin{corollary}\label{app-inf-neumann}
Let $\Omega\subset \rn$ be an exterior Lipschitz domain, $n\ge 2$.
 Let $A\in VMO(\rn)$ satisfy
 $$\fint_{B(x_0,r)}|A-I_{n\times n}|\,dx\le \frac{C}{r^\delta}$$
for some $\delta>0$, all $r>1$ and all $x_0\in\rn$.

(i) There exists $\epsilon>0$ such that for  $1<p<3+\epsilon$ when $n\ge 3$, $ 1<p<4+\epsilon$  when $n=2$,
there exists $C>1$ such that for all $f\in \dot{W}^{1,p}(\Omega)$ it holds that
$$\|\nabla \LV_N^{-1/2}f\|_{L^p(\Omega)}\le C\|f\|_{L^p(\Omega)}.$$
Moreover, for $\frac{3+\epsilon}{2+\epsilon}<p<3+\epsilon$ when $n\ge 3$, $  \frac{4+\epsilon}{3+\epsilon}<p<4+\epsilon$  when $n=2$, it holds for all $f\in \dot{W}^{1,p}(\Omega)$ that
$$C^{-1}\|\nabla f\|_{L^p(\Omega)}\le \|\LV_N^{1/2}f\|_{L^p(\Omega)}\le C\|\nabla f\|_{L^p(\Omega)}.$$

 (ii) If $\Omega$ is $C^1$, then conclusion of (i) holds for $\epsilon=\infty$.
\end{corollary}
\begin{proof}
By the proof of Corollary \ref{app-inf-dirichlet}, $p_\LV=\infty$ for this case.
The conclusion then follows from Theorem \ref{app-neumann}.
\end{proof}

Finally, let us provide an application to the inhomogeneous Dirichlet/Neumann problem on exterior domains.
{For a given domain $\Omega$, we consider the Dirichlet problem
\begin{eqnarray}\label{dp}
\begin{cases}
\LV_D u=-\mathrm{div} f& \, {\text{in}}\, \Omega,\\
u=0 & \, {\text{on}}\, \partial\Omega,
\end{cases}
\end{eqnarray}
and the Neumann problem
\begin{eqnarray}\label{np}
\begin{cases}
\LV_N u=-\mathrm{div} f& \, {\text{in}}\, \Omega,\\
\nu\cdot A\nabla u=\nu\cdot f & \, {\text{on}}\, \partial\Omega.
\end{cases}
\end{eqnarray}
As usual, the equations are understood in the weak sense.
For a given $p\in (1,\infty)$ and $f\in L^p(\Omega)$, we seek for solution
$u$ with appropriate boundary conditions to the above equations that satisfies
$$\|\nabla u\|_{L^p(\Omega)}\le C(p)\|f\|_{L^p(\Omega)},$$
where $C(p)$ depends on $p$ but not $f$.
}

Let us recall some (incomplete) literature here. Lions and Magenes \cite{lm68} studied these problems
in smooth domains, see also Grisvard \cite{gr85} for the case $p=2$ in less smooth cases.
For the Dirichlet Laplacian $\Delta_D$, Jerison and Kenig \cite{jk95} obtained optimal result on bounded $C^1$ and Lipschitz domains for the problem $(D_p)$. For the Neumann  Laplacian $\Delta_N$, Zanger \cite{zan00}
obtained optimal result on bounded Lipschitz domains for the problem $(N_p)$, about the same time, Fabes, Mendez and Mitrea \cite{fmm98} systematically treated both the Dirichlet and Neumann  Laplacian on bounded Lipschitz domains by boundary integral methods.
For elliptic operators with discontinuous coefficients (VMO coefficients, small BMO coefficients, partially BMO coefficients),
we refer the reader to \cite{aq02,by05,bw04,bw05,DF96,dk10,ge12,kr07,shz05} for $C^{1,1}$, $C^1$, Lipschitz or Reifenberg domains.

Previously, Amrouche, Girault and Giroire \cite{agg97} obtained sharp estimate on exterior $C^{1,1}$ domains
for the Dirichlet Laplacian $\Delta_D$ and Neumann Laplacian $\Delta_N$, by using theory of weighted Sobolev spaces.
Here, by viewing the operator $\nabla \LV_D^{-1}\mathrm{div}$ as
$$\nabla \LV_D^{-1/2}\circ \LV_D^{-1/2}\mathrm{div}=\nabla \LV_D^{-1/2} \circ (\nabla \LV_D^{-1/2})^\ast,$$ where $(\cdot)^\ast$ denotes the conjugate operator (the Neumann case is similar), we deduce from
Corollaries \ref{app-inf-dirichlet} and \ref{app-inf-neumann} the following application. For simplicity of notions, we denote $\min\{a,b\}$ by $a\wedge b$.
\begin{corollary}\label{app-inhomogenous}
Let $\Omega\subset \rn$ be an exterior Lipschitz domain, $n\ge 2$.
 Let $A\in VMO(\rn)$ satisfy
 $$\fint_{B(x_0,r)}|A-I_{n\times n}|\,dx\le \frac{C}{r^\delta}$$
for some $\delta>0$, all $r>1$ and all $x_0\in\rn$.

(i) For $n=3$, for any $p\in (n',n)$ and any $f\in L^p(\Omega)$, $(D_p)$ has a unique solution
$u\in \dot{W}^{1,p}_0(\Omega)$ that satisfies $\|\nabla u\|_{L^p(\Omega)}\le C\|f\|_{L^p(\Omega)}$.

(ii) For $n\ge 4$, there exists $\epsilon>0$ such that for any $p\in ((n\wedge (3+\epsilon))',n\wedge (3+\epsilon))$ and any $f\in L^p(\Omega)$, $(D_p)$ has a unique solution
$u\in \dot{W}^{1,p}_0(\Omega)$ that satisfies $\|\nabla u\|_{L^p(\Omega)}\le C\|f\|_{L^p(\Omega)}$.

(iii) For $n\ge 4$, if $\Omega$ is $C^1$, then for any $p\in (n',n)$ and any $f\in L^p(\Omega)$, $(D_p)$ has a unique solution $u\in \dot{W}^{1,p}_0(\Omega)$ that satisfies $\|\nabla u\|_{L^p(\Omega)}\le C\|f\|_{L^p(\Omega)}$.

(iv) There exists $\epsilon>0$, such that for $p\in ((3+\epsilon)', 3+\epsilon)$ when $n\ge 3$ and
$p\in ((4+\epsilon)', 4+\epsilon)$ when $n=2$, and any $f\in L^p(\Omega)$, $(N_p)$ has a unique solution $u\in \dot{W}^{1,p}(\Omega)$ that satisfies $\|\nabla u\|_{L^p(\Omega)}\le C\|f\|_{L^p(\Omega)}$.

(v) If $\Omega$ is $C^1$, then for any $p\in (1,\infty)$ and any $f\in L^p(\Omega)$, $(N_p)$ has a unique solution $u\in \dot{W}^{1,p}(\Omega)$ that satisfies $\|\nabla u\|_{L^p(\Omega)}\le C\|f\|_{L^p(\Omega)}$.

\end{corollary}
From the results in \cite{agg97,fmm98,ge12,jk95,shz05,zan00} and the example from Remark \ref{unbound-dirichlet}, we know the range of $p$ is sharp. For general operators
$\LV_D$, $\LV_N$, we may use Theorems \ref{main-dirichlet}, \ref{main-neumann}, \ref{app-dirichlet} and
\ref{app-neumann} instead of Corollaries \ref{app-inf-dirichlet} and \ref{app-inf-neumann}, which will not be  repeated here.

\subsection*{Acknowledgments}
\addcontentsline{toc}{section}{Acknowledgments} \hskip\parindent
The authors wish to thank the referee for the valuable comments which improve 
the quality of the paper. R. Jiang was partially supported by NNSF of China (11922114 \& 11671039) and NSF of Tianjin (20JCYBJC01410),  F.H. Lin
was in part supported by the National Science Foundation Grant DMS1955249.

\noindent Renjin Jiang \\
\noindent Academy for Multidisciplinary Studies, Capital Normal University, Beijing, 100048\\
\& \\
\noindent  Center for Applied Mathematics, Tianjin University, Tianjin 300072, China\\
\noindent {rejiang@cnu.edu.cn}

\

\noindent Fanghua Lin\\
\noindent Courant Institute of Mathematical Sciences, 251 Mercer Street, New York, NY 10012, USA \\
\noindent linf@cims.nyu.edu

\end{document}